\documentclass[12pt,a4paper,psamsfonts]{amsart}

\usepackage{fancyhdr}
\usepackage{appendix}
\usepackage{amssymb,amscd,amsxtra,calc}
\usepackage{mathrsfs}
\usepackage{cmmib57}
\usepackage{multirow}
\usepackage[all]{xy}
\usepackage{longtable}
\usepackage[colorlinks,linkcolor=blue,anchorcolor=blue,citecolor=green]{hyperref}
\usepackage{tikz}

\theoremstyle{plain}
    \newtheorem{thm}{Theorem}[section]
    
    \newtheorem{claim}[thm]{Claim}
     \newtheorem{conjecture}[thm]{Conjecture}
    \newtheorem{corollary}[thm]{Corollary}
    
    \newtheorem{lemma}[thm]{Lemma}
    \newtheorem{proposition}[thm]{Proposition}
    \newtheorem{question}[thm]{Question}
    \newtheorem{theorem}[thm]{Theorem}

\theoremstyle{definition}

    \newtheorem{notation}[thm]{Notation}
    \newtheorem*{notation*}{Notation and Terminology}
      
    \newtheorem{remark}[thm]{Remark}
    \newtheorem*{ack}{Acknowledgments}
\theoremstyle{remark}

\newcommand{\mstriangle}[1]{
\begin{tikzpicture}[x=0.3cm,y=0.3cm]
\draw (-0.4,-0.433) -- (1.4,-0.433);
\draw (-0.2,-0.7794) -- (0.7,0.7794);
\draw (1.2,-0.7794) -- (0.3,0.7794);
\end{tikzpicture}
}
\newcommand{\mssharp}[1]{
\begin{tikzpicture}[x=0.3cm,y=0.3cm]
\draw (-0.8,-0.5) -- (0.8,-0.5);
\draw (-0.8,0.5) -- (0.8,0.5);
\draw (-0.5,-0.8) -- (-0.5,0.8);
\draw (0.5,-0.8) -- (0.5,0.8);
\end{tikzpicture}
}

\makeatletter

\newcommand{\Rmnum}[1]{\expandafter\@slowromancap\romannumeral #1@}
\makeatother

\begin{document}

\title[Strictly nef]
{Strictly nef divisors on singular threefolds}

\author{Guolei Zhong}


{

\address
{
\textsc{National University of Singapore,
Singapore 119076, Republic of Singapore
}}
\email{zhongguolei@u.nus.edu}
\begin{abstract}
Let $X$ be a normal projective threefold with mild singularities, and $L_X$ a strictly nef $\mathbb{Q}$-divisor on $X$.
First, we show the ampleness of $K_X+tL_X$ with sufficiently large $t$ if either the Kodaira dimension $\kappa(X)\neq 0$ or the augmented irregularity $q^{\circ}(X)\neq 0$.
Second, we show that, if $(X,\Delta)$ is a projective klt threefold pair with  the anti-log canonical divisor $-(K_X+\Delta)$ being strictly nef, then $X$ is rationally connected.
\end{abstract}



\maketitle
\tableofcontents

\section{Introduction}
We work over the field $\mathbb{C}$ of complex numbers. 
A $\mathbb{Q}$-Cartier divisor $L$ on  a normal projective variety $X$ is said to be \textit{strictly nef}, if $L\cdot C>0$ for every curve $C$ on $X$. 
It has been known that strictly nef but non-ample divisors do exist.
The following famous conjecture, which is proposed by Campana and Peternell, leads people to learn how far strictly nef anti-canonical divisors are from being ample.

\begin{conjecture}\label{conj_CP}(\cite[Problem 11.4]{CP91})
Let $X$ be a smooth projective variety with the anti-canonical divisor $-K_X$ being strictly nef.
Then $X$ is Fano, i.e., $-K_X$ is ample.	
\end{conjecture}

In the 1990s, Conjecture \ref{conj_CP} was intensively studied in lower dimensions and
it has been confirmed in dimension $\le 3$ (cf.~\cite{Mae93} and \cite{Ser95}).
Subsequently, Conjecture \ref{conj_CP} was further extended to the singular case and a positive answer  for threefolds with only canonical singularities  has been obtained; see \cite[Theorem 3.9]{Ueh00}.

Recently, studying the positivity of tangent bundles, Li, Yang and Ou showed that a smooth projective variety with the  anti-canonical divisor being strictly nef is rationally connected (cf.~\cite[Theorem 1.2]{LOY19}). 
Since a projective klt pair $(X,\Delta)$ with $-(K_X+\Delta)$  being ample is rationally connected (cf.~\cite{Zha06}), 
this nice result \cite[Theorem 1.2]{LOY19} provides evidence for  Conjecture \ref{conj_CP}  in all dimensions. 
It has also been a long history to seek the numerical criterion for a variety to be rationally connected (cf.~\cite{Zha06}).  
Motivated by Conjecture \ref{conj_CP} and  \cite[Theorem 1.2]{LOY19}, we ask the following question.

\begin{question}\label{gene_conj}
Let $(X,\Delta)$ be a projective klt pair with the anti-log canonical divisor $-(K_X+\Delta)$ being strictly nef.
Will $X$ be rationally connected?
\end{question}

Question \ref{gene_conj} has a positive answer (in all dimensions) when $\Delta=0$ and $X$ is smooth (cf.~\cite[Theorem 1.2]{LOY19}).
It extends Conjecture \ref{conj_CP} to a (weaker) version for  normal projective varieties  
and also relates the positivity on $X$ to its underlying geometric properties.
This is the initial point of the paper.

In \cite{Ser95}, to study Conjecture \ref{conj_CP}, Serrano also posed the following generalized conjecture  and verified it for Abelian varieties, Gorenstein surfaces and  smooth threefolds with some possible exceptions (cf.~\cite[Proposition 1.4 and Theorems 2.3, 4.4]{Ser95}).  
\begin{conjecture}\label{conj_serreno}(cf.~\cite[Question 0.1]{Ser95})
Let $X$ be a smooth projective variety, and $L_X$ a strictly nef line bundle on $X$.
Then $L_X+\varepsilon K_X$ is ample for sufficiently small rational numbers $\varepsilon>0$.
\end{conjecture}

One decade later,  Campana, Chen and Peternell  confirmed Conjecture \ref{conj_serreno} in dimension three with the only (possible) exception that $X$ is Calabi-Yau and $L_X\cdot c_2(X)=0$ (cf.~\cite[Theorem 0.4]{CCP08}). 
Ample divisors being strictly nef,  Conjecture \ref{conj_serreno} can also be regarded as a weak analogue of Fujita's Conjectures. 
In the spirit of Conjecture \ref{conj_serreno}, we consider the (pair case of) singular varieties; see \cite[Corollary D]{LP20A} for a partial answer when $\dim X=3$.
\begin{question}\label{ques_serrano}
Let $(X,\Delta)$ be a projective klt pair, and  
 $L_X$  a strictly nef $\mathbb{Q}$-divisor on $X$.
Will $K_X+\Delta+tL_X$ be ample for sufficiently large $t\gg 1$?
\end{question}

Note  that, a positive answer to Question \ref{ques_serrano} will give a solution to Question \ref{gene_conj}; see \cite{Zha06}. 
In this paper, we shall study Questions \ref{gene_conj}
and  \ref{ques_serrano} for the case $\dim X\le 3$, with the main results Theorem  \ref{intro_main_prop} and Theorem 
\ref{main_Goren_ter_3fold} (or Theorem \ref{main_theorem_Goren_ter_3fold}, a more general form).

As a warm up, the following proposition 
answers  Questions \ref{gene_conj} and \ref{ques_serrano} affirmatively for the surface case. 
We shall give an alternative proof of Proposition \ref{main_thm_surface} in Section \ref{section2}.
\begin{proposition}(\cite[Corollary 1.8]{HL20}; see Proposition \ref{prop_Q-Goren_surface} for a further extension)\label{main_thm_surface}
Questions \ref{ques_serrano}  and \ref{gene_conj} have  positive answers when $\dim X=2$.
\end{proposition}

Now  we consider Question \ref{ques_serrano}  when $X$ is a singular threefold.   
Theorem \ref{main_Goren_ter_3fold} below generalizes \cite[Theorem 0.4]{CCP08} to the singular setting. 
Terminal singularities being isolated canonical, we shall show Theorem \ref{main_Goren_ter_3fold}  in a more general category of singularities (cf.~Theorem \ref{main_theorem_Goren_ter_3fold}). 
The condition on singularities is  to avoid flips when we run the minimal model program (MMP for short); see  
Lemmas \ref{lem_del14}, \ref{lem_canonical_terminal}, and Remark \ref{rem_composition_conic}.

\begin{theorem}(see Theorem \ref{main_theorem_Goren_ter_3fold} for a more general version)\label{main_Goren_ter_3fold}
Let $X$ be a $\mathbb{Q}$-factorial Gorestein terminal projective threefold, and $L_X$ a strictly nef $\mathbb{Q}$-divisor on $X$.
Suppose that either $\kappa(X)\neq 0$ or $q^{\circ}(X)\neq 0$.
Then $K_X+tL_X$ is ample for sufficiently large $t\gg 1$.
\end{theorem}

Recall that the \textit{augmented irregularity} $q^\circ(X)$ is defined as the maximum of the irregularities $q(\widetilde{X}):=h^1(\widetilde{X},\mathcal{O}_{\widetilde{X}})$, where $\widetilde{X}\to X$ runs over the (finite) quasi-\'etale covers of $X$ (cf.~\cite[Definition 2.6]{NZ10}). 


From our viewpoints, it is not satisfactory that the conditions of Theorem \ref{main_Goren_ter_3fold}  seem quite restrictive.  
We believe that, one can consider  Question \ref{gene_conj} itself (avoiding Question \ref{ques_serrano}) by analyzing the universal covering of $X$; see~the recent series papers  \cite{Cao19},  \cite{CH19}, \cite{CCM19}, \cite{Wan20} and \cite{MW21}.


Looking for a special section of the Albanese morphism, we  can show such $X$ in Question \ref{gene_conj} has vanishing irregularity  (cf.~Theorem \ref{thm_num_trivial}), and  finally give a positive answer to  Question \ref{gene_conj} for the threefold case. 
This also extends the second assertion of \cite[Theorem 1.2]{LOY19} from the case of  $X$ being smooth to the klt  pair case when $\dim X=3$.
\begin{theorem}\label{intro_main_prop}
Let $(X,\Delta)$ be a projective klt threefold pair with the anti-log canonical divisor $-(K_X+\Delta)$ being strictly nef.
Then $X$ is rationally connected.
\end{theorem}

\begin{ack}
The author would like to thank Professor De-Qi Zhang for numerous inspiring discussions.
The author would also like to thank Professor Vladimir Lazi\'c
for pointing out \cite[Corollary 1.8]{HL20}, Professor Wenhao Ou for useful suggestions, and Professor Juanyong Wang for valuable discussions  to improve the paper.
The author is supported by President's Scholarships of NUS.
\end{ack}

\section{Preliminaries}\label{section2}
Throughout this paper, we refer to \cite[Chapter 2]{KM98} for different kinds of singularities.
By a \textit{projective klt (resp. canonical, dlt) pair} $(X,\Delta)$, we mean that $X$ is a normal projective variety, $\Delta\ge 0$, $K_X+\Delta$ is $\mathbb{Q}$-Cartier, and  $(X,\Delta)$ has only klt (resp. canonical, dlt) singularities.
For the convenience of readers, we shall recall  several early results 
and extend some of them to the singular setting 
to keep pace with the present situation. 

\begin{theorem}(\cite[Theorem 3.9]{Ueh00})\label{thm_canonical_k}
Let $X$ be a canonical 3-fold with strictly nef anti-canonical divisor $-K_X$.
Then $-K_X$ is ample.
\end{theorem}

\begin{proposition}\label{prop_strict_uniruled}
Let $(X,\Delta)$ be a projective pair such that $-(K_X+\Delta)$ is strictly nef.
Then $X$ is uniruled.	
\end{proposition}
\begin{proof}
Suppose the contrary that $X$ is not uniruled.
Then $K_X$ is pseudo-effective as a Weil divisor (cf. \cite{BDPP13}) by considering a resolution of $X$.
So $-\Delta=-(K_X+\Delta)+K_X$ is pseudo-effective (as a 
Weil divisor) and thus $\Delta=0$.
This in turn implies the strict nefness of $-K_X$, a contradiction to the pseudo-effectivity of $K_X$.	
\end{proof}

\begin{lemma}(cf.~\cite[Lemma 1.1]{Ser95})\label{lem_strict_nef_k}
Let $(X,\Delta)$ be a projective klt pair, and $L_X$  a strictly nef $\mathbb{Q}$-divisor on $X$.
Then $(K_X+\Delta)+tL_X$ is  strictly nef for every $t>2m\dim X$ where $m$ is the Cartier index of $L_X$.
\end{lemma}

\begin{proof}
By the cone theorem (cf.~\cite[Theorem 3.7]{KM98}), every curve $C$ in $X$ is numerically equivalent to a linear combination 
$M+\sum a_iC_i$ (a finite sum), where  $M$ is  pseudo-effective such that $M\cdot (K_X+\Delta)\ge 0$ and the $C_i$'s are (integral) rational curves satisfying $0<C_i\cdot (-K_X-\Delta)\le 2\dim X$	 with $a_i>0$.
Since $mL_X$ is Cartier and strictly nef,  we have $L_X\cdot M\ge 0$, and $L_X\cdot C_i\ge \frac{1}{m}$ for all $i$.
Therefore, for each $C_i$, 
\begin{align} \tag{\dag}\label{eq_lem_strict}
	((K_X+\Delta)+tL_X)\cdot C_i\ge tL_X\cdot C_i-2\dim X>0.
\end{align}
If $(K_X+\Delta)\cdot C\ge 0$, then $((K_X+\Delta)+tL_X)\cdot C>0$ since $L_X\cdot C>0$.
If $(K_X+\Delta)\cdot C<0$, then the decomposition of $C$ contains some $C_i$; hence our lemma follows from (\ref{eq_lem_strict}).
\end{proof}

\begin{lemma}\label{lem-big-ample}
Let $(X,\Delta)$ be a projective klt pair, 
and $L_X$  a strictly nef $\mathbb{Q}$-divisor on $X$.	
Suppose that $a(K_X+\Delta)+bL_X$ is big for some $a,b\ge 0$.
Then $(K_X+\Delta)+tL_X$ is ample for sufficiently large $t$. 
\end{lemma}

\begin{proof}
If $a=0$, then $L_X$ is big. 
If $a\neq 0$, then $(K_X+\Delta)+\frac{b}{a}L_X$ is big.
In both cases, it follows from  Lemma \ref{lem_strict_nef_k} that $(K_X+\Delta)+tL_X$ is strictly nef and big for $t\gg 1$, noting that nef divisors are always pseudo-effective.
Then $2((K_X+\Delta)+tL_X)-(K_X+\Delta)$ is also nef and big.
By the base-point-free theorem (cf.~\cite[Theorem 3.3]{KM98}), some multiple of $(K_X+\Delta)+tL_X$ defines a morphism $X\to\mathbb{P}^N$, which
 is finite by the projection formula.
Therefore, $(K_X+\Delta)+tL_X$ is ample.
Our lemma is proved.
\end{proof}

We denote by $\overline{\textup{NE}}(X)$ (resp. $\overline{\textup{ME}}(X)$) the \textit{Mori cone} (resp. \textit{movable cone}) of $X$ (cf.~\cite{BDPP13}).
The following lemma characterizes the situation when $K+tL$ is not big.
\begin{lemma}\label{lem-not-big-some}
Let $(X,\Delta)$ be a projective klt threefold pair, 
and $L_X$  a strictly nef $\mathbb{Q}$-divisor on $X$.	
Suppose $(K_X+\Delta)+uL_X$ is not big for some rational number $u>6m$ with $m$  the Cartier index of $L_X$.
Then $(K_X+\Delta)^3=(K_X+\Delta)^2\cdot L_X=(K_X+\Delta)\cdot L_X^2=L_X^3=0$. 	
Moreover, there exists a  class $0\neq \alpha\in\overline{\textup{ME}}(X)$ such that 
$(K_X+\Delta)\cdot\alpha=L_X\cdot\alpha=0$. 
\end{lemma}

\begin{proof}
Let $u'=u-6m>0$, $D_1=(K_X+\Delta)+(\frac{u'}{2}+6m)L_X$ and $D_2=\frac{u'}{2}L_X$.	
Then $(K_X+\Delta+uL_X)^3=(D_1+D_2)^3=0$.
Since both $D_1$ and $D_2$ are nef (cf.~Lemma \ref{lem_strict_nef_k}), the vanishing $D_1^3=D_1^2\cdot D_2=D_1\cdot D_2^2=D_2^3=0$ yields the first part of our lemma.

Since the nef divisor $K_X+\Delta+uL_X$ is not big, there exists a class $\alpha\in\overline{\textup{ME}}(X)$ such that $(K_X+\Delta+uL_X)\cdot\alpha=0$ (cf.~\cite[Theorem 2.2 and the remarks therein]{BDPP13}). 
Suppose that $L_X\cdot\alpha\neq 0$.  
Then $L_X\cdot\alpha>0$ and thus $(K_X+\Delta)\cdot\alpha<0$.
By the cone theorem (cf.~\cite[Theorem 3.7]{KM98}), we have 
$\alpha=M+\sum a_iC_i$ 
where $M$ is a class lying in $\overline{\textup{NE}}(X)_{K_X+\Delta\ge 0}$, $a_i>0$, and the $C_i$ are extremal rational curves with $0<-(K_X+\Delta)\cdot C_i\le 6$.
Now, for each $i$, the intersection $(K_X+\Delta+uL_X)\cdot C_i>0$, a contradiction to the choice of $\alpha$.
So we have $(K_X+\Delta)\cdot \alpha=L_X\cdot\alpha=0$.
\end{proof}

Now we give an alternative proof of Proposition \ref{main_thm_surface} (cf.~\cite[Corollary 1.8]{HL20}) by applying a recent result on almost strictly nef divisors on surfaces. 
Recall that a $\mathbb{Q}$-Cartier divisor $L_X$ on a  normal projective variety $X$ is said to be \textit{almost strictly nef}, if there are a birational morphism $\pi:X\to Y$ to some  projective variety $Y$ and a strictly nef divisor $L_Y$ on $Y$ such that $L_X=\pi^*L_Y$ (cf.~\cite[Definition 1.1]{CCP08}).

\begin{proof}[\textup{\textbf{Proof of Proposition  \ref{main_thm_surface}}}]
Replacing $L_X$ by a multiple, we may assume that $L_X$ is  Cartier.
Taking a minimal resolution $\pi:\widetilde{X}\to X$, we see that $L_{\widetilde{X}}:=\pi^*L_X$ is almost strictly nef. 
By \cite[Theorem 20]{Cha20}, $K_{\widetilde{X}}+tL_{\widetilde{X}}$ is big for $t\gg 1$.
Then $K_X+tL_X$	 is big (as a Weil divisor) for  $t\gg 1$ (cf.~\cite[Lemma 4.10]{FKL16}). 
So $K_X+\Delta+tL_X$ is big (as a Cartier divisor).
Since $K_X+\Delta+tL_X$ is strictly nef for $t\gg 1$ (cf.~Lemma \ref{lem_strict_nef_k}),
the first part of our theorem follows from Lemma \ref{lem-big-ample}.   
The second part  follows from \cite{Zha06}.
\end{proof}

In what follows, we slightly generalizes Proposition \ref{main_thm_surface}   to the following (not necessarily normal)  $\mathbb{Q}$-Gorenstein surface case (cf.~\cite[Conjecture 1.3]{CCP08}). 
This is also mentioned at the end of \cite[Section 2]{Ser95}.

For a $\mathbb{Q}$-factorial normal projective variety $X$ and a prime divisor $S\subseteq X$,  the \textit{canonical divisor} $K_S\in\textup{Pic}(S)\otimes\mathbb{Q}$ is defined by  
$K_S:=\frac{1}{m}(mK_X+mS)|_S$, where $m\in\mathbb{N}$ is the smallest positive integer such that both $mK_X$ and $mS$ are Cartier divisors on $X$. 

\begin{proposition}\label{prop_Q-Goren_surface}
Let $(X,\Delta)$ be a $\mathbb{Q}$-factorial dlt threefold pair,  $S$ a prime divisor on $X$, and  $L_S$ a strictly nef  divisor on $S$.
Then $K_S+tL_S$ is ample for sufficiently large $t$.
\end{proposition}
\begin{proof}
Since $X$ is $\mathbb{Q}$-factorial, there exists some positive integer $m$ such that $m(K_X+S)$ is Cartier. 
Let $j:T\to S$ be the normalization, and $f:R\to T$  a minimal resolution.
In view of \cite[Proof of Theorem 2.3]{Ser95}, we only need to reprove the following injection	in \cite[Claim 4 of Proof of Theorem 2.3]{Ser95} for sufficiently large $r\gg 1$:
$$f_*\mathcal{O}_R(rmK_R)\hookrightarrow j^*\mathcal{O}_S(rmK_S).$$ 
Since $(X,\Delta)$ is a dlt pair,  $X$ is Cohen-Macaulay (cf.~\cite[Theorems 5.10 and 5.22]{KM98}).
By \cite[Lemma 5-1-9]{KMM87}, there is a natural injective homomorphism
$$\omega_T^{[m]}:=(\omega_T^{\otimes m})^{\vee\vee}=\mathcal{O}_T(mK_T)\hookrightarrow j^*\mathcal{O}_X(m(K_X+S))=j^*\mathcal{O}_S(mK_S).$$
On the other hand, it follows from \cite[Theorem (2.1)]{Sak84} that
$$f_*\mathcal{O}_R(f^*(mK_T))\cong\mathcal{O}_T(mK_T).$$
Here, our $mK_T$  is only a Weil divisor. 
Replacing $m$ by a multiple, we write $f^*(mK_T)=mK_R+\Gamma$ with $\Gamma$ being an effective (integral)  divisor (cf.~\cite[(4.1)]{Sak84}).
Then 
$$f_*\mathcal{O}_R(mK_R)\subseteq f_*(\mathcal{O}_R(mK_R)\otimes\mathcal{O}_R(\Gamma))=f_*\mathcal{O}_R(f^*(mK_T))\cong\mathcal{O}_T(mK_T)\hookrightarrow j^*\mathcal{O}_S(mK_S).$$
Hence, we get the inclusion as desired and  our proposition is proved. 
\end{proof}
As an application, we  extend \cite[Proposition 3.1]{Ser95} 
 to the singular setting (cf.~Proposition \ref{prop-q-effective}), 
which will be crucially used to prove the ampleness of $K+tL$.
\begin{proposition}\label{prop-q-effective}
Let $(X,\Delta)$ be a $\mathbb{Q}$-factorial dlt threefold pair,  and  $L_X$ a strictly nef $\mathbb{Q}$-divisor on $X$.
Suppose that there is a non-zero effective divisor numerically equivalent to $aL_X+bK_X$ for some  $a$ and $b$. 
Then $K_X+tL_X$ is ample for  $t\gg 1$.
\end{proposition}

\begin{proof}
Suppose  the contrary. 
By Lemmas \ref{lem_strict_nef_k} and \ref{lem-big-ample},  $K_X+tL_X$ is not big (but nef) for some  $t>6m$ with $m$ the Cartier index of $L_X$. 
Then Lemma \ref{lem-not-big-some} gives us that
$$K_X^3=K_X^2\cdot L_X=K_X\cdot L_X^2=L_X^3=0.$$
Pick an effective divisor $\sum n_iF_i\equiv aL_X+bK_X$.
Then  $(K_X+tL_X)^2\cdot (aL_X+bK_X)=0$.
Since $K_X+tL_X$ is nef for all $t\gg 1$,  
we have  $F_i\cdot L_X^2=F_i\cdot K_X\cdot L_X=F_i\cdot K_X^2=0$ for each prime component $F_i$. 
By Proposition \ref{prop_Q-Goren_surface}, $K_{F_1}+rL_X|_{F_1}$ is ample for $r\gg 1$.  
Applying the Hodge index theorem, we have 
$$0<L_X|_{F_1}\cdot (K_{F_1}+rL_X|_{F_1})=F_1\cdot L_X\cdot (K_X+F_1+rL_X)=F_1^2\cdot L_X.$$
Now that $F_i$ and $F_j$ are distinct prime divisors, we deduce the following
$$0=F_1\cdot (aL_X+bK_X)\cdot L_X=F_1\cdot\sum  n_iF_i\cdot L_X\ge n_1F_1^2L_X>0,$$
a contradiction. 
So $K_X+tL_X$ is ample for  $t\gg 1$ and our proposition is thus proved.
\end{proof}


To end this section, we extend a formula on conic bundles to  our singular case.
Recall that a \textit{contraction} of $X$ is a surjective morphism with connected fibres.
A contraction $\pi:X\to S$ is said to be \textit{elementary} if the relative Picard number $\rho(X/S)=1$.
We refer readers  to Remark \ref{rem_composition_conic} for a further extension  when $\rho(X/S)\ge 2$.
\begin{lemma}(cf.~\cite[4.11]{Miy81})\label{lem_conic_miy}
Let $X$ be a normal projective threefold with at worst isolated klt singularities.
Suppose that $\pi:X\to S$ is an elementary  $K_X$-negative contraction  onto a normal projective surface $S$.
Denote by $D_1$ the one-dimensional part of the discriminant locus of $\pi$ (over which, $\pi$ is not smooth).
Then $\pi_*K_X^2\equiv -(4K_S+D_1)$.
\end{lemma}

\begin{proof}
Since $\rho(X/S)=1$, by the  canonical bundle formula (cf.\,e.g.~\cite[Theorem 0.2]{Amb05}), there exists some  $\Delta_S\ge 0$ on $S$ such that $(S,\Delta_S)$ is klt and thus $S$ has only rational singularities (cf.~\cite[Theorem 5.22]{KM98}).
Then, $S$ is $\mathbb{Q}$-factorial (cf.~\cite[Proposition 17.1]{Lip69}), and 
we only need to check the equality  on very ample curves  on $S$. 
By the cone theorem (cf.~\cite[Theorem 3.7]{KM98}), $\pi$ is equi-dimensional.
By \cite[Theorem 5.10]{KM98}, $X$ is Cohen-Macaulay.
Let $D_0:=(\textup{Sing}~S)\cup\pi(\textup{Sing}~X)$.
Then, $\pi|_{X\backslash\pi^{-1}(D_0)}$ is flat (cf.~\cite[Theorem 23.1 and its Corollary]{Mat89}) and thus $\pi|_{X\backslash\pi^{-1}(D_0)}$ is a usual conic bundle  (cf.~\cite[(1.5)]{Sar82}).  
Let $T$ be a very ample curve on $S$ which can be assumed to be smooth, avoid $D_0$, and intersect with $D_1$ transversally. 
Let $F:=\pi^{-1}(T)=\pi^*(T)$, which is smooth.
Then 
$$K_X^2\cdot F=(K_X|_F)^2=(K_F-F|_F)^2=K_F^2-2K_F\cdot (F|_F)=K_F^2+4T^2.$$ 
Note that the last equality is due to the adjunction.
Since $F$ is a (not necessarily minimal) ruled surface over $T$ with $D_1
\cdot T$ degenerate  (reducible) fibres, we have
$$K_F^2=-4(K_S\cdot T+T^2)-D_1\cdot T.$$ 
As a result, we have $\pi_*(K_X^2)\cdot T=K_X^2\cdot F=-(4K_S+D_1)\cdot T$.
So our lemma is proved.
\end{proof}

\section{Reduction to uniruled threefolds with $q^{\circ}(X)=0$}\label{section_reduction_q=0_uniruled}
In Sections \ref{section_reduction_q=0_uniruled} $\sim$ \ref{sec_pf_thm1.6}, we  prove Theorem \ref{main_Goren_ter_3fold}.
More precisely, we  show the following theorem which is of a  more general setting. 
For the proofs, we follow the ideas in \cite{CCP08}, though in our case, the appearance of singularities will make things more complicated.

\begin{theorem}(cf.~Theorem \ref{main_Goren_ter_3fold})\label{main_theorem_Goren_ter_3fold}
Let $X$ be a normal projective threefold with only canonical singularities, 
and $L_X$ a strictly nef divisor on $X$.
Suppose one of the following holds.
\begin{enumerate}
\item The Kodaira dimension $\kappa(X)\ge 1 $;
\item The augmented irregularity $q^{\circ}(X)>0$ and $X$ is $\mathbb{Q}$-factorial; or
\item $X$ is uniruled, and  has only isolated $\mathbb{Q}$-factorial Gorenstein canonical singularities. 
\end{enumerate}
Then $K_X+tL_X$ is ample for sufficiently large $t\gg 1$.
\end{theorem} 

Theorem \ref{main_theorem_Goren_ter_3fold} is obtained as the union of Theorem \ref{kappa>=1}, Corollary \ref{coro_q0>0}, Remark \ref{remark_conclude_excep_tocurve} and Theorem \ref{thm_3fold_surface_curve}.
In this section, we shall prove Theorem \ref{main_theorem_Goren_ter_3fold} (1) and (2).
\begin{theorem}(cf.~Theorem \ref{main_theorem_Goren_ter_3fold} (1))\label{kappa>=1}
Let $(X,\Delta)$ be a projective canonical pair, and $L_X$ a strictly nef $\mathbb{Q}$-divisor on $X$.
Suppose that the Iitaka dimension $\kappa(a(K_X+\Delta)+bL_X)\ge \dim X-2$ for some $a,b\ge 0$.
Then $(K_X+\Delta)+tL_X$ is ample for sufficiently large $t$.
\end{theorem}

\begin{proof}
If $\dim X\le 2$, then our theorem follows from Proposition  \ref{main_thm_surface}.
So we may assume $\dim X-2\ge 1$.
By Lemma \ref{lem_strict_nef_k}, $K_X+\Delta+tL_X$ is strictly nef for $t\gg 1$. 
With $b$ replaced by $b/a$ (if $a\neq 0$), we may further assume  $0\le a\le 1$.
Let $\pi:X\dashrightarrow Y$ be the Iitaka fibration (with $\dim Y\ge 1$).
Resolving the indeterminacy of $\pi$ and the singularities of $X$, we get the induced morphisms $\pi_1:\widetilde{X}\to X$ and $\pi_2:\widetilde{X}\to Y$ such that  $\widetilde{X}$ is smooth and
\begin{equation}\label{equ_1}
	K_{\widetilde{X}}+\widetilde{\Delta}=\pi_1^*(K_X+\Delta)+\sum a_iE_i
\end{equation}
with $\widetilde{\Delta}:=(\pi_1)_*^{-1}\Delta$ and $a_i\ge 0$, noting that $(X,\Delta)$ is a projective canonical pair.
Moreover, there exists an ample $\mathbb{Q}$-divisor $H$ on $Y$ such that
\begin{equation}\label{equ_2}
\pi_1^*(a(K_X+\Delta)+bL_X)\sim_{\mathbb{Q}}\pi_2^*H+E_0.	
\end{equation}
Here, some multiple $nE_0$ is the fixed component of $|n\pi_1^*(a(K_X+\Delta)+bL_X)|$ (which defines $\pi_2$) and hence $E_0\ge 0$. 
Let $L_{\widetilde{X}}:=\pi_1^*L_X$, being almost strictly nef (cf.~Section \ref{section2}) on $\widetilde{X}$.
Then for a general fibre $F$ of $\pi_2$, the restriction $L_{\widetilde{X}}|_F$ is almost strictly nef.

If $\kappa(a(K_X+\Delta)+bL_X)=\dim X$, then our theorem follows from Lemma \ref{lem-big-ample}.
If $\kappa(a(K_X+\Delta)+bL_X)=\dim X-1$, then  $L_{\widetilde{X}}|_F$ is clearly ample.
If $\kappa(a(K_X+\Delta)+bL_X)=\dim X-2$, then $\dim F=2$; in this case, it follows from \cite[Theorem 20]{Cha20} that $K_F+tL_{\widetilde{X}}|_F$ is big for $t\gg 1$ and hence $(K_{\widetilde{X}}+\widetilde{\Delta}+tL_{\widetilde{X}})|_{F}$ is big for $t\gg 1$.
Let us fix the ample $\mathbb{Q}$-divisor $H$ on $Y$ in Equation (\ref{equ_2}).

In any event, 
$K_{\widetilde{X}}+\widetilde{\Delta}+tL_{\widetilde{X}}+\pi_2^*H$
is big for $t\gg 1$: 
 fixing one $t$ such that $(K_{\widetilde{X}}+\widetilde{\Delta}+tL_{\widetilde{X}})|_{F}$ is big, we see that  $A_k:=(K_{\widetilde{X}}+\widetilde{\Delta}+tL_{\widetilde{X}})+k\pi_2^*H$ is big for $k\gg 1$ (cf.~e.g. \cite[Lemma 3.23]{KM98}).
Since $K_{\widetilde{X}}+\widetilde{\Delta}+tL_{\widetilde{X}}=\pi_1^*(K_X+\Delta+tL_X)+\sum a_iE_i$ is   pseudo-effective, our $A_1=K_{\widetilde{X}}+\widetilde{\Delta}+tL_{\widetilde{X}}+\pi_2^*H=\frac{1}{k}(A_k+(k-1)(K_{\widetilde{X}}+\widetilde{\Delta}+tL_{\widetilde{X}}))$ is big.

Let $N:=(1-a)(K_{\widetilde{X}}+\widetilde{\Delta})+(t-b)L_{\widetilde{X}}$, which is pseudo-effective for $t\gg 1$ (cf.~Lemma \ref{lem_strict_nef_k} and Equation (\ref{equ_1})).
Following from Equations (\ref{equ_1}) and (\ref{equ_2}), we have
\begin{align*}
2(K_{\widetilde{X}}+\widetilde{\Delta}+tL_{\widetilde{X}})&=A_1+(a(K_{\widetilde{X}}+\widetilde{\Delta})+bL_{\widetilde{X}}-\pi_2^*H)+N\\
&=
A_1+\pi_1^*(a(K_X+\Delta)+bL_X)+\sum aa_iE_i-\pi_2^*H+N\\
&\sim_{\mathbb{Q}} A_1+E_0+\sum aa_iE_i+N(\ge A_1+N),
\end{align*}
which is big for $t\gg 1$.
Thus, the push-forward $K_X+\Delta+tL_X$ is also big for $t\gg1$ (cf.~\cite[Lemma 4.10]{FKL16}).
So our theorem follows from Lemma \ref{lem-big-ample}.
\end{proof}

With  the same proof of \cite[Theorem 3.1]{CCP08}  after replacing \cite[Theorem 0.3, Lemma 1.5, Propositions 1.7, 1.8]{CCP08} by 
Theorem \ref{thm_canonical_k}, Lemma \ref{lem-big-ample} and Propositions \ref{prop_Q-Goren_surface}, \ref{prop-q-effective}, respectively, we  get the theorem below (cf.~\cite[Theorem 8.1]{LP20B}).  
For readers' convenience, we  include a detailed proof here. 
\begin{theorem}(cf.~\cite[Theorem 3.1]{CCP08})\label{thm_alb}
Let $X$ be a  normal projective threefold with only canonical singularities, and $L_X$ a strictly nef $\mathbb{Q}$-divisor on $X$.
Suppose there exists a non-constant morphism $\pi:X\to A$ to an abelian variety (this is the case when $q(X)>0$).
Then $K_X+tL_X$ is numerically equivalent to a non-zero effective divisor for  some $t\gg 1$. 
If $X$ is further assumed to be $\mathbb{Q}$-factorial, then $K_X+tL_X$ is ample for $t\gg 1$.
\end{theorem}

\begin{proof}
Replacing $L_X$ by a multiple, we may assume that $L_X$ is Cartier.
Let $m$ be the Cartier index of $K_X$, $D_t:=2m(K_X+tL_X)$ and $\mathcal{F}_t:=\pi_*\mathcal{O}_X(D_t)$. 
Fix one $t>6m$ such that $D_t$ is strictly nef and  $D_t|_{F}$ is ample for a general fibre $F$ of $\pi$ (cf.~Proposition \ref{prop_Q-Goren_surface} and \cite[Theorem 1.37]{KM98}). 
Then it is easy to verify that $D_t$ is $\pi$-big  (cf.~\cite[Definition 3.22]{KM98}). 
Since $D_t-K_X=(2m-1)(K_X+tL_X)+tL_X$ is  $\pi$-big and $\pi$-nef, by the relative base-point-free theorem (cf.~\cite[Theorem 3.24]{KM98}), $nD_t$ is $\pi$-free  for all $n\gg 1$.
Hence, with $m$ replaced by a multiple, we may further assume that $D_t$ is $\pi$-free (i.e., $\pi^*\mathcal{F}_t\to \mathcal{O}_X(D_t)$ is surjective), and thus $\mathcal{F}_t$ is a non-zero sheaf.

Let $\hat{A}=\textup{Pic}^0(A)$ be the dual abelian variety and $\mathcal{P}$  the normalized Poincar\'e line bundle (in the sense that both $\mathcal{P}|_{A\times\{\hat{0}\}}$ and $\mathcal{P}|_{\{0\}\times\hat{A}}$ are trivial).
Denote by $p_A$ and $p_{\hat{A}}$ the projections of $A\times\hat{A}$ onto $A$ and $\hat{A}$.
For any ample line bundle $\hat{M}$ on $\hat{A}$, we define the isogeny $\phi_{\hat{M}}:\hat{A}\to A$  by $\phi_{\hat{M}}(\hat{a})=t_{\hat{a}}^*\hat{M}^{\vee}\otimes\hat{M}$, and
let $M:=(p_{A})_*(p_{\hat{A}}^*\hat{M}\otimes\mathcal{P})$ be the vector bundle on $A$.
Recall that $\phi_{\hat{M}}^*(M^{\vee})\cong\oplus_{h^0(\hat{M})}\hat{M}$ (cf.~\cite[Proposition 3.11 (1)]{Muk81}). 
\begin{claim}\label{claim_fm}
$H^i(A,\mathcal{F}_t\otimes M^{\vee})=0$ for any  $t\gg1$ and any  ample line bundle $\hat{M}$ on $\hat{A}$. 
\end{claim}
Suppose the claim for the time being.
Then it follows from \cite[Theorem 1.2 and Corollary 3.2]{Hac04} that
we have a chain of inclusions
$V^0(\mathcal{F}_t)\supseteq V^1(\mathcal{F}_t)\supseteq\cdots\supseteq V^n(\mathcal{F}_t)$, 
where $V^i(\mathcal{F}_t):=\{P\in\textup{Pic}^0(A)~|~h^i(A,\mathcal{F}_t\otimes P)\neq 0\}$.
Recall that $\mathcal{F}_t$ is a  non-zero sheaf by the choice of our $m$. 
If  $V^i(\mathcal{F}_t)=\emptyset$ for all $i$, then the Fourier-Mukai transform of $\mathcal{F}_t$ is zero, noting that $\mathcal{P}|_{A\times\{P\}}\cong P$ and 
$R^i(p_{\hat{A}})_*(p_A^*\mathcal{F}_t\otimes\mathcal{P})_P=H^i(A,\mathcal{F}_t\otimes P)$ 
for any $P\in\textup{Pic}^0(X)$.  
By \cite[Theorem 2.2]{Muk81},  $\mathcal{F}_t$ is a zero sheaf, which is absurd. 
So $V^0(\mathcal{F}_t)\neq\emptyset$ and thus
$h^0(X,D_t+\pi^*P)\neq 0$ for some $P\in\textup{Pic}^0(A)$.
If $D_t+\pi^*P\sim 0$, then $-K_X$ is strictly nef and hence $X$ is Fano by Theorem \ref{thm_canonical_k}; thus  $q(X)=0$, contradicting the existence of $\pi$. 
So  $D_t\equiv D_t+\pi^*P\sim E$ for some non-zero effective divisor $E$, 
and the first part of our theorem is proved.
The second part of our theorem follows from Proposition \ref{prop-q-effective}. 
\par \vskip 1pc \noindent
\textbf{Proof of Claim \ref{claim_fm}.}
Fix an ample divisor $H$ on $A$. 
Since $D_{t}$ is  $\pi$-big, our $D_t+\pi^*H$ is big, noting that $D_t$ is not only $\pi$-big, but also nef (and hence pseudo-effective) (cf.~\cite[Lemma 3.23]{KM98}).
Therefore, $D_t+\pi^*(H+P)$ is nef and big for any $P\in\textup{Pic}^0(A)$. 
Since $X$ has at worst canonical singularities and 
$$D_t+\pi^*(H+P)=K_X+\frac{2m-1}{2m}(D_t+\pi^*(H+P))+\frac{1}{2m}\pi^*(H+P)+tL_X=:K_X+\widetilde{D}$$
with $\widetilde{D}$ being a nef and big $\mathbb{Q}$-Cartier Weil divisor, it follows from  Kawamata-Viehweg Vanishing Theorem (cf. e.g. \cite[Theorem 2.70]{KM98}) that 
$$H^j(A,\mathcal{F}_t\otimes H\otimes P)=H^j(X,D_t+\pi^*H+\pi^*P)=0$$
for all $j\ge1$ (where $\mathcal{F}_t=\pi_*D_t$). 
The first equality is due to $R^j\pi_*(D_t+\pi^*H+\pi^*P)=0$ for all $j\ge1$ (cf.~\cite[Remark 1-2-6]{KMM87}).
With the same terminology as in \cite[Proof of Theorem 3.1]{CCP08},  the sheaf $\mathcal{F}_t\otimes H$ is said to be $IT^0$ for all ample line bundles $H$.

Let $\hat{M}$ be any ample line bundle on $\hat{A}$, and $\phi_{\hat{M}}:\hat{A}\to A$  the isogeny defined  in the beginning of the proof. 
Let $\hat{\pi}:\hat{X}:=X\times_A\hat{A}\to \hat{A}$ be the base change with the induced map $\varphi:\hat{X}\to X$ being \'etale.
Then we have $K_{\hat{X}}=\varphi^*K_X$ and $L_{\hat{X}}:=\varphi^*L_X$ is strictly nef on $\hat{X}$.
Let $\mathcal{G}_t:=\hat{\pi}_*\varphi^*D_t$.
With the same argument as above, we see that $\mathcal{G}_t\otimes\hat{M}$ is $IT^0$ for any ample line bundle $\hat{M}$ on $\hat{A}$.
Since $\phi_{\hat{M}}$ is flat and $\pi$ is projective, we have
\begin{align*}
&~\phi_{\hat{M}}^*(\mathcal{F}_t\otimes M^{\vee})=\phi_{\hat{M}}^*(\pi_*D_t\otimes M^{\vee})=\phi_{\hat{M}}^*\pi_*(D_t\otimes \pi^*M^{\vee})=\hat{\pi}_*\varphi^*(D_t\otimes \pi^*M^{\vee})\\
&=\hat{\pi}_*(\varphi^*D_t\otimes\hat{\pi}^*\phi_{\hat{M}}^*M^{\vee})=\hat{\pi}_*(\varphi^*D_t\otimes\hat{\pi}^*(\oplus \hat{M}))=\oplus(\hat{\pi}_*\varphi^*D_t\otimes\hat{M})=\oplus(\mathcal{G}_t\otimes\hat{M}).
\end{align*}
Therefore, $\phi_{\hat{M}}^*(\mathcal{F}_t\otimes M^{\vee})$ is $IT^0$, i.e., for any $j\ge 1$ and $\hat{P}\in\textup{Pic}^0(\hat{A})=A$, 
$$H^j(\hat{A},\phi_{\hat{M}}^*(\mathcal{F}_t\otimes M^{\vee})\otimes\hat{P})=0.$$
Since $\phi_{\hat{M}}^*$ is finite, taking $\hat{P}:=\mathcal{O}_{\hat{A}}$, we see $H^j(A,\mathcal{F}_t\otimes M^{\vee}\otimes (\phi_{\hat{M}})_*\mathcal{O}_{\hat{A}})=0$.
Now that $\mathcal{O}_A$ is a direct summand of $(\phi_{\hat{M}})_*\mathcal{O}_{\hat{A}}$ (cf.~\cite[Proposition 5.7]{KM98}),  our claim is proved.
\end{proof}

We note that when $K_X$ is pseudo-effective, \cite[Theorem 8.1]{LP20B} shows a stronger version of Theorem \ref{thm_alb}. 
Let us end up this section with the following corollary. 
\begin{corollary}(cf.~Theorem \ref{main_theorem_Goren_ter_3fold} (2))\label{coro_q0>0}
Let $X$ be a $\mathbb{Q}$-factorial normal projective threefold with at worst canonical singularities, and $L_X$ a strictly nef $\mathbb{Q}$-divisor on $X$.
Suppose that the augmented irregularity $q^\circ(X)>0$.
Then $K_X+tL_X$ is ample for sufficiently large $t$.
\end{corollary}

\begin{proof}
Let $\pi:X'\to X$ be a quasi-\'etale cover such that $q(X')>0$, and $L_{X'}:=\pi^*L_X$.
By \cite[Proposition 5.20]{KM98}, $X'$  has only canonical singularities. 
By Theorem \ref{thm_alb}, our $K_{X'}+tL_{X'}\equiv E>0$ for some $t\gg 1$.
So $K_X+tL_X$ is weakly numerically equivalent to $\pi_*(E)>0$ (cf.~\cite[Definition 2.2]{MZ18}). 
Since $X$ is $\mathbb{Q}$-factorial, weak numerical equivalence of divisors coincide with numerical equivalence (cf.~\cite[Lemma 3.2]{Zha16}); hence $K_X+tL_X\equiv \pi_*(E)$. 
By Proposition \ref{prop-q-effective}, our corollary follows.
\end{proof}

\section{The first step of the MMP}\label{section_4}
In this section, we shall run the first step of the MMP for the proof of Theorem \ref{main_theorem_Goren_ter_3fold} (3). 
Let us  recall the following special  case of \cite[Theorem 2.2]{Del14},  which generalizes  \cite[Theorem 4]{Cut88} to the  isolated canonical singularities to avoid the small contractions.

\begin{lemma}(cf.~\cite[Theorem 2.2]{Del14})\label{lem_del14}
Let $X$ be a normal projective threefold with only isolated $\mathbb{Q}$-factorial Gorenstein canonical singularities. 
Let $\varphi:X\to Y$ be a birational $K_X$-negative contraction  of an extremal face $R$ whose fibres are at most one-dimensional.
Then the following assertions hold.
\begin{enumerate}
\item 
The exceptional locus $E:=\textup{Exc}(\varphi)$ is a disjoint union of prime divisors, 
$\varphi$ is a composition of divisorial contractions mapping exceptional divisors onto curves, and $Y$ has only isolated canonical singularities.
\item There exists a finite subset $T\subseteq Y$  such that $Y\backslash T\subseteq Y_{\textup{reg}}$,  $\text{codim}~\varphi^{-1}(T)\ge 2$, $X\backslash \varphi^{-1}(T)\subseteq X_{\textup{reg}}$ and 
$\varphi|_{X\backslash\varphi^{-1}(T)}:X\backslash \varphi^{-1}(T)\to Y\backslash T$
is the simultaneous blow-up of   smooth curves.
In particular, $K_X=\varphi^*(K_Y)+\sum E_i$ with $E_i$ being exceptional, and $\textup{Sing}~(\sum E_i)\subseteq \varphi^{-1}(\textup{Sing}~\sum\varphi(E_i))$.
\item Let $f\subseteq X$ be an irreducible curve such that $[f]\in R$. Then
$K_X\cdot f=E\cdot f=-1$.
\end{enumerate}
\end{lemma}

Based on Lemma \ref{lem_del14}, we establish the following lemma on an elementary contraction (of an extremal ray), which is a key to reduce the isolated canonical singularities to terminal singularities via the terminalization.

\begin{lemma}\label{lem_canonical_terminal}
Let $X$ be a normal projective threefold with only isolated $\mathbb{Q}$-factorial Gorenstein canonical singularities. 
Suppose that $\pi:X\to Y$ is a birational elementary  contraction of a $K_X$-negative extremal ray.
Then we have the following commutative diagram
\[\xymatrix{\widetilde{X}\ar[r]^{\widetilde{\pi}}\ar[d]_{\tau_1}&\widetilde{Y}\ar[d]^{\tau_2}\\
X\ar[r]_{\pi}&Y
}
\]
such that the following assertions hold.
\begin{enumerate}
\item $\tau_1$ is the (crepant) terminalization and $\widetilde{X}$ has only $\mathbb{Q}$-factorial Gorenstein terminal singularities.
\item $\widetilde{\pi}$ is an elementary  contraction of a $K_{\widetilde{X}}$-negative extremal ray, which is   divisorial.
\item $Y$ has at worst $\mathbb{Q}$-factorial isolated canonical singularities.
\end{enumerate}
Suppose further that either $\pi$ maps the exceptional divisor $E:=\textup{Exc}(\pi)$ to a curve, or $\widetilde{\pi}$ maps the exceptional divisor $\widetilde{E}:=\textup{Exc}(\widetilde{\pi})$ to a single point.
Then we further have:
\begin{enumerate}
\item[(4)] $\widetilde{E}\cap\textup{Exc}(\tau_1)=\emptyset$ and $\widetilde{E}=\tau_1^{-1}(E)\cong E$;
\item[(5)]  $\tau_2$ is also crepant; and in particular, 
\item[(6)] if $\pi(E)$ is a curve, then $X$ is the blow-up of $Y$ along a local complete intersection curve $C$, $Y$ is  Gorenstein which is smooth around $C$, and $E$ is  Cartier; moreover, $\pi_*(-E|_E)=\pi(E)=C$.
\end{enumerate}
\end{lemma}

\begin{proof}
Let $\tau_1:\widetilde{X}\to X$ be the terminalization where $\widetilde{X}$ has only $\mathbb{Q}$-factorial Gorenstein terminal singularities.
Let $\widetilde{C}\subseteq\widetilde{X}$ be an irreducible curve such that $\tau_1(\widetilde{C})$ is $\pi$-contracted.
Then $K_{\widetilde{X}}\cdot \widetilde{C}=K_X\cdot\tau_1(\widetilde{C})<0$.
Therefore, there is an extremal ray $\widetilde{R}\in\overline{\textup{NE}}(\pi\circ\tau_1)\subseteq\overline{\textup{NE}}(\widetilde{X})$ such that $K_{\widetilde{X}}\cdot\widetilde{R}<0$.
Let $\widetilde{\pi}:\widetilde{X}\to \widetilde{Y}$ be the  contraction of $\widetilde{R}$.
By the rigidity lemma, 
 $\pi\circ\tau_1$ factors through $\widetilde{\pi}$.
So $\widetilde{\pi}$ is birational and thus divisorial (cf. \cite[Theorems 4 and 5]{Cut88}); hence
(1) and (2) are proved. 
By Lemma \ref{lem_del14}, $\pi$ is  divisorial; hence 
(3) follows from \cite[Proposition 3.36]{KM98}, noting that non-terminal points of $Y$  are contained in the images of the non-terminal points of $X$.
Let $\widetilde{F}$ be any curve contracted by $\widetilde{\pi}$.
Then $K_{\widetilde{X}}\cdot\widetilde{F}<0$ implies that $F:=\tau_1(\widetilde{F})$ is still a curve  and thus contracted by $\pi$.
Hence, $\tau_1(\widetilde{E})\subseteq E$ (being irreducible), which implies that $\tau_1(\widetilde{E})=E$.

From now on, we assume that either $\pi(E)$ is a curve or $\widetilde{\pi}(\widetilde{E})$ is a single point.
\textbf{We claim that $E$ does not contain any non-terminal point of $X$.}
Suppose the contrary.
Since $X$ is $\mathbb{Q}$-factorial, it follows from \cite[Corollary 2.63]{KM98} that $\textup{Exc}(\tau_1)$ is of pure codimension one.
Therefore, $\tau_1^*(E)=\widetilde{E}+F$ with $F$ being $\tau_1$-exceptional, noting that $\widetilde{E}$ is an irreducible divisor by (2).
Since $\widetilde{X}$ is also $\mathbb{Q}$-factorial, we take a curve $B$ in $\widetilde{E}\cap F$.
If $\widetilde{E}$ is mapped onto a single point, then $\widetilde{\pi}(B)$ is a point; hence $0>K_{\widetilde{X}}\cdot B=\tau_1^*K_X\cdot B=0$, a contradiction.
Suppose that $E$ is mapped onto a curve. 
Then with a similar argument, 
$\widetilde{\pi}(\widetilde{E})=\widetilde{\pi}(B)$ is  a curve (i.e., $B$ is not $\widetilde{\pi}$-contracted); hence $\widetilde{\pi}(\widetilde{E})=\widetilde{\pi}(B)$ is  contracted by $\tau_2$.
But then, $\tau_2(\widetilde{\pi}(\widetilde{E}))$ is  a single point on $Y$, contradicting  the assumption that  $\pi(E)$ is a curve.
So our claim holds and (4) is proved.
(5) follows from the the diagram (for some positive number $a$):
$$K_{\widetilde{Y}}=\widetilde{\pi}_*(K_{\widetilde{X}}+a\widetilde{E})=\widetilde{\pi}_*\tau_1^*\pi^*K_Y=\tau_2^*K_Y.$$
If $\pi(E)$ is a curve, so is $\widetilde{\pi}(\widetilde{E})$;  hence $K_{\widetilde{Y}}$ is Cartier
(cf.~\cite[Lemma 3]{Cut88}).
By \cite[Theorem 4]{Cut88} and (1),  $X$ is the blow-up of a local complete intersection curve on $Y$ and $Y$ is smooth around $C$. 
Since $K_{\widetilde{Y}}=\tau_2^*K_Y$, our $K_Y$ is also Cartier: indeed, since $\tau_2$ is birational, the function field $K(\widetilde{Y})$  coincides with $K(Y)$; hence the local defining  equation for $K_{\widetilde{Y}}$ on $\widetilde{Y}$ is also that for $K_Y$. 
Applying Lemma \ref{lem_del14} (2) to $\pi$, we see that $E$ is also Cartier. 
Then  
 $\mathcal{O}_X(-E)|_E\cong\mathcal{O}_E(1)$. 
Since $-E|_E$ is $\pi|_E$-ample, we have   $q(-E|_E)+G\equiv H_E$ for some $q>0$, some fibres $G$, and some very ample curve $H_E$ on $E$.
By the intersection theory and Lemma \ref{lem_del14}, $\deg (H_E\to C)=q$, so $(6)$ is proved.
\end{proof}

\begin{remark}
If $\pi$ contracts $E$ to a point in Lemma  \ref{lem_canonical_terminal}, then  (4) is not necessarily true, since $\widetilde{\pi}$ in this case will  possibly contract $\widetilde{E}$ to a curve, and then $\widetilde{E}$ is the blow-up of $E$ along some point; in other words, $E$ may contain some non-terminal point of $X$.
\end{remark}

Propositions \ref{contr_curve_elementary} $\sim$ \ref{prop_contr_surface_elementary} below describe the first step of the MMP which is either a Fano contraction or a divisorial contraction mapping the exceptional divisor to a single point.

\begin{proposition}\label{contr_curve_elementary}
Let $X$ be a normal  projective threefold with only canonical singularities, and $L_X$ a strictly nef divisor on $X$.
Suppose  $\rho(X)\le 2$ (this is the case when $X$ admits a $K_X$-negative  elementary contraction to a curve).
Then $K_X+tL_X$ is ample for  $t\gg 1$.
\end{proposition}
\begin{proof}
If view of  Theorem \ref{thm_canonical_k}, we may assume $\rho(X)=2$ and $K_X$ is not parallel to $L_X$.  
Hence, the Mori cone $\overline{\textup{NE}}(X)$ has only two extremal rays. 
In particular, some linear combination $aK_X+bL_X$ is strictly positive on $\overline{\textup{NE}}(X)\backslash \{0\}$ and thus ample (cf.~\cite[Theorem 1.18]{KM98}).
By Lemma \ref{lem-big-ample}, our proposition is proved.
\end{proof}

\begin{proposition}\label{prop_Exc_surface_point}
Let $X$ be a $\mathbb{Q}$-factorial normal  projective threefold with only  canonical singularities and $L_X$ a strictly nef divisor. 
Suppose  there is a divisorial contraction $\pi:X\to X'$ mapping the exceptional divisor to a point.	 
Then $K_X+tL_X$ is ample for  $t\gg1$.
\end{proposition}

\begin{proof}
In view of Lemma \ref{lem-big-ample}, we may assume that $K_X+tL_X$ is not big.
By Lemma \ref{lem-not-big-some}, $K_X^i\cdot L_X^{3-i}=0$ for any $0\le i\le 3$.
Write $K_X-aE=\pi^*K_{X'}$ and $L_X+bE=\pi^*L_{X'}$ with $L_{X'}$ being  $\mathbb{Q}$-Cartier  (cf.~\cite[Theorem 3-2-1]{KMM87}) and $E$ being  $\pi$-exceptional.
Since $-E$ is $\pi$-ample (cf.~\cite[Lemma 2.62]{KM98}) and $L_X$ is strictly nef, our $b>0$. 
By the projection formula, $L_{X'}$ is strictly nef.
Note that 
$\pi^*(K_{X'}+tL_{X'})=(K_X+tL_X)+(bt-a)E$;  
hence for  $t\gg 1$, our $K_{X'}+tL_{X'}$ is strictly nef.

We claim that $K_{X'}+tL_{X'}$ is big (hence ample by Lemma \ref{lem-big-ample}). 
Suppose the claim for the time being.
Fixing some $m\gg 1$, we choose a  smooth member $A\in |m(K_{X'}+tL_{X'})|$ which is ample.
Let $D_X:=bK_X+aL_X$ and $D_{X'}:=bK_{X'}+aL_{X'}$.
Then  $D_X=\pi^*(D_{X'})$.
Since $E$ is mapped to a point, $D_X^2\cdot(K_X+tL_X)=0$ implies $D_{X'}^2\cdot(K_{X'}+tL_{X'})=0$.
Similarly, $D_{X'}\cdot (K_{X'}+tL_{X'})^2=D_X\cdot (K_X+tL_X+(bt-a)E)^2=0$.  
So $D_{X'}^2\cdot A=D_{X'}\cdot  A^2=0$.
By \cite[Lemma 3.2]{Zha16}, we have $D_{X'}\equiv 0$. 
Then, $D_X\equiv0$ and   $-K_X$ is proportional to $L_X$, which is strictly nef.
As a result, $X$ is Fano (cf.~Theorem \ref{thm_canonical_k}) and our result follows.

It remains to show the bigness of $K_{X'}+tL_{X'}$.
Since $\pi(E)$ is  a single point, we have
$$0\le(K_{X'}+tL_{X'})^3=\pi^*(K_{X'}+tL_{X'})\cdot (K_X+tL_X)^2=(bt-a)(K_X+tL_X)^2\cdot E.$$
Here, $-E$ is $\pi$-ample. 
Thus, $(L_X|_E)^2=L_X^2\cdot E=(\pi^*L_{X'}-bE)^2\cdot E=(-bE|_E)^2>0$.
By the  Nakai-Moishezon criterion (cf.~\cite[Theorem 1.37]{KM98}),  $L_X|_E$ is ample.
So the last item is positive for  $t\gg 1$.
In particular, $K_{X'}+tL_{X'}$ is big.
\end{proof}

\begin{proposition}
\label{prop_contr_surface_elementary}
Let $X$ be a normal projective threefold with at worst isolated $\mathbb{Q}$-factorial  canonical singularities,  and $L_X$ a strictly nef $\mathbb{Q}$-divisor on $X$.
Suppose that $X$ admits an elementary  $K_X$-negative contraction  $\pi:X\to S$ to a normal projective surface.
Then $K_X+tL_X$ is ample for sufficiently large $t$.
\end{proposition}

\begin{proof}
Let us assume that $L_X$ is Cartier after replacing $L_X$ by a multiple. 
By Corollary \ref{coro_q0>0}, Lemmas \ref{lem_strict_nef_k} and \ref{lem-big-ample}, we may assume $q^\circ(X)=0$, and $K_X+tL_X$ is nef but not big for  $t\gg 1$. 
Note that we can further assume $K_X$ not parallel to $L_X$ (cf.~Theorem \ref{thm_canonical_k}).  
By the canonical bundle formula (cf.~e.g. \cite[Theorem 0.2]{Amb05}), $S$ has only rational singularities  and thus is $\mathbb{Q}$-factorial (cf.~\cite[Proposition 17.1]{Lip69}).
Let $u:=\frac{-K_X\cdot f}{L_X\cdot f}>0$ (which is a rational number) where $f$ is a general fibre of $\pi$.
By the cone theorem (cf.~\cite[Theorem 3.7]{KM98}), there exists a $\mathbb{Q}$-Cartier divisor $M$ such that
$K_X+uL_X=\pi^*M$ with $M\not\equiv 0$.  
Then, Lemma \ref{lem-not-big-some} gives us 
\begin{align*}
0=(K_X-\pi^*M)^3=-3K_X^2\cdot(K_X+uL_X)+3K_X\cdot(\pi^*M)^2=3(K_X\cdot f)M^2,
\end{align*}
which implies  $M^2=0$.  
Fix a rational number $t>6$ such that  $K_X+tL_X$ is strictly nef.  
By Lemma \ref{lem-not-big-some},  there exists  $\alpha\in\overline{\textup{ME}}(X)$ such that $K_X\cdot\alpha=L_X\cdot \alpha=0$.
Then $M\cdot \gamma=0$ with $0\not\equiv\gamma:=\pi_*\alpha\in\overline{\textup{ME}}(S)$.
Since $\dim S=2$, our $\gamma$ is nef.
Since $M^2=0$ but $M\not\equiv0$, we have $\gamma^2=0$ and thus $\gamma$ is parallel to $M$ by the Hodge index theorem. 
Hence, either $M$  is nef or $-M$ is nef.
Denote by $M_1=M$ (resp. $-M$) if $M$ is nef (resp. $-M$ is nef). 
Replacing $M_1$ by a multiple, we may assume that $M_1$ is a line bundle.

\begin{claim}\label{claim_m-k-pseu}
$N:=M-K_S$	is pseudo-effective.
\end{claim}

\noindent
\textbf{Proof of Claim \ref{claim_m-k-pseu}.}
For any irreducible curve $\ell\subseteq S$, it follows from Lemma \ref{lem_conic_miy} that
\begin{align*}
 0\le  u^2L_X^2\cdot\pi^*(\ell)&=(\pi^*M-K_X)^2\cdot\pi^*(\ell)=-2\pi^*M\cdot K_X\cdot \pi^*(\ell)	+K_X^2\cdot\pi^*(\ell)\\
 &=-2(K_X\cdot f)M\cdot \ell-(4K_S+D_1)\cdot \ell=(4N-D_1)\cdot \ell.
\end{align*}
Here,  $D_1$ denotes the one-dimensional part of the discriminant locus of  $\pi$.  
Then $4N-D_1$ is nef and hence
$N=\frac{1}{4}((4N-D_1)+D_1)$ is pseudo-effective.

\par \vskip 1pc \noindent

We come back to the proof of Proposition \ref{prop_contr_surface_elementary}. 
Let $\tau:\widetilde{S}\to S$ be a minimal resolution, and  $\widetilde{M_1}:=\tau^*M_1$. 
Since  $M_1\cdot M=0$ and $M_1\cdot N\ge 0$ (cf.~Claim \ref{claim_m-k-pseu}), our $K_{\widetilde{S}}\cdot \widetilde{M_1}=K_S\cdot M_1\le 0$. 
Then, for any positive integer $n$, applying the Riemann-Roch formula to $n\widetilde{M_1}$ and noting that $q(\widetilde{S})=q(S)=q(X)=0$, we get the following inequality
\begin{align}\label{eq_RR}
\begin{split}
h^0(S,nM_1)&=h^0(\widetilde{S},n\widetilde{M_1})=-\frac{n}{2}\widetilde{M_1}\cdot K_{\widetilde{S}}+\chi(\mathcal{O}_{\widetilde{S}})+h^1(\widetilde{S},n\widetilde{M_1})-h^2(\widetilde{S},n\widetilde{M_1})\\
&\ge 1-h^2(\widetilde{S},n\widetilde{M_1}).
\end{split}
\end{align}

\textbf{We claim that $h^2(\widetilde{S},n\widetilde{M_1})=h^0(K_{\widetilde{S}}-n\widetilde{M_1})=0$ for $n\gg 1$.}  
Fixing an ample divisor $\widetilde{H}$ on $\widetilde{S}$, we have $\widetilde{M_1}\cdot \widetilde{H}>0$ by the Hodge index theorem (noting that $\widetilde{M_1}^2=M_1^2=0$).
For each $n$, if there is an effective divisor $Q_n$  on $\widetilde{S}$ such that $K_{\widetilde{S}}-n\widetilde{M_1}\sim Q_n$, then 
$$K_{\widetilde{S}}\cdot \widetilde{H}=(n\widetilde{M_1}+Q_n)\cdot\widetilde{H}\ge n\widetilde{M_1}\cdot\widetilde{H}.$$
The left hand side of the above inequality being bounded, our claim is thus proved.

Therefore, it follows from Equation (\ref{eq_RR}) that $M_1$ is numerically equivalent to an effective $\mathbb{Q}$-divisor, which is non-zero by our assumption in the beginning of the proof.
This implies that 
 $K_X+uL_X$ (or $-K_X-uL_X$) is  numerically equivalent to  a non-zero effective divisor. 
By Proposition \ref{prop-q-effective}, our proposition is proved.
\end{proof}

The following theorem generalizes Proposition \ref{prop_contr_surface_elementary} to the non-elementary case.

\begin{theorem}\label{thm_contr_surface_conic}
Let $X$ be a normal projective threefold with only isolated $\mathbb{Q}$-factorial Gorenstein canonical singularities,  and $L_X$ a strictly nef divisor on $X$.
Suppose that $X$ admits an equi-dimensional  $K_X$-negative contraction (of an extremal face) $\pi:X\to S$ onto a normal projective surface such that $K_X+uL_X=\pi^*M$ for some $u\in\mathbb{Q}$ and some $\mathbb{Q}$-Cartier divisor $M$ on $S$.
Then $K_X+tL_X$ is ample for $t\gg 1$.
\end{theorem}

\begin{remark}\label{rem_composition_conic}
With the same assumption as in Theorem \ref{thm_contr_surface_conic},  applying \cite[Proof of Proposition 3.4]{Rom19} and Lemma  \ref{lem_del14}, we 
get the following MMP for $X$ over $S$, noting that the Gorenstein condition on $X$ can descend to each $X_i$ (cf.~Lemma \ref{lem_canonical_terminal} (6)).
\[\xymatrix{X=:X_0\ar[r]^{\phi_0}\ar[drrr]_{\pi_0:=\pi}&X_1\ar[r]^{\phi_1}&\cdots\ar[r]^{\phi_{n-1}}&X_n\ar[d]^{\pi_n}\\
&&&S
}
\]
Here, each $\phi_i$ is an elementary $K_{X_i}$-negative divisorial contraction mapping the exceptional divisor  onto a curve (cf.~Lemma \ref{lem_del14}), each $X_i$ is a normal projective threefold with only isolated $\mathbb{Q}$-factorial Gorenstein canonical singularities  (cf.~Lemma \ref{lem_canonical_terminal}), and $\pi_n$ is an elementary conic fibration in the sense that the generic fibre is a smooth plane conic. 

With this kept in mind, we can extend the formula in Lemma \ref{lem_conic_miy} to the non-elementary case in our situation.
Let $\phi:=\phi_{n-1}\circ\cdots\circ\phi_0$.
By Lemma \ref{lem_del14} (2) and Lemma \ref{lem_canonical_terminal} (6),  $K_X=\phi^*K_{X_n}+\sum E_i$ with $E_i$  being pairwise disjoint, and $\phi_*(-E_i|_{E_i})=\phi(E_i)$.
Then for any divisor $H$ on $X_n$, we have 
\begin{align*}
	\phi_*(K_X^2)\cdot H&=(\phi^*K_{X_n}\cdot\phi^*K_{X_n}+2\sum\phi^*K_{X_n}\cdot E_i+\sum E_i|_{E_i})\cdot \phi^*H\\
	&=K_{X_n}^2\cdot H-\sum \phi(E_i)\cdot H
\end{align*}
where $\phi(E_i)$ are pairwise disjoint  curves, and all of them are not $\pi_n$-contracted. 
Combining the above numerical equivalence with Lemma \ref{lem_conic_miy}, we have 
$\pi_*(K_X^2)\equiv (\pi_n)_*(K_{X_n}^2)-\sum a_i\pi(E_i)=:-(4K_S+D)$,
where each $a_i>0$, and $D$ is the sum of the one-dimensional part of the discriminant locus of $\pi_n$ and $\sum a_i\pi(E_i)$; hence, $D$ is  effective.
\end{remark}

\begin{proof}[Proof of Theorem \ref{thm_contr_surface_conic}]
The  proof is completely the same as Proposition \ref{prop_contr_surface_elementary} after replacing Lemma \ref{lem_conic_miy} by Remark \ref{rem_composition_conic}.
\end{proof}

\begin{remark}\label{remark_conclude_excep_tocurve}
We shall finish the proof of  Theorem \ref{main_theorem_Goren_ter_3fold} (3) in Section \ref{sec_pf_thm1.6}.
In view of Lemma \ref{lem_del14} and Propositions \ref{contr_curve_elementary}  $\sim$ \ref{prop_contr_surface_elementary},  we only need to verify the case when all the elementary $K_X$-negative  contractions (of extremal rays) are divisorial contractions, mapping the exceptional divisors to curves, 
which is Theorem \ref{thm_3fold_surface_curve}.
\end{remark}

\section{Proof of Theorems \ref{main_Goren_ter_3fold} and  \ref{main_theorem_Goren_ter_3fold} (3)}\label{sec_pf_thm1.6}

The whole section is devoted to proving the following theorem.
We shall always stick to Notation \ref{not_bir_surface_curve} and apply the induction on the Picard number $\rho(X)$. 

\begin{theorem}(cf.~\cite[Proposition 5.2]{CCP08})\label{thm_3fold_surface_curve}
Let  $X$ be a  $\mathbb{Q}$-factorial normal Gorenstein projective uniruled threefold with only isolated   canonical singularities, 
and $L_X$ a strictly nef divisor.
Suppose  all the elementary $K_X$-negative extremal contractions  are divisorial, mapping the exceptional divisors to curves.
Then $K_X+tL_X$ is ample for $t\gg 1$.
\end{theorem}

\begin{notation}\label{not_bir_surface_curve}
\begin{enumerate}
\item Let $X$ be a normal projective threefold with at worst isolated  $\mathbb{Q}$-factorial Gorenstein canonical singularities, and $L_X$ a strictly nef divisor.
\item In view of Theorem \ref{main_theorem_Goren_ter_3fold} (2), we may assume  $q^\circ(X)=0$.
\item Let $\varphi_i:X\to X_i$ be the contraction of the $K_X$-negative extremal ray $\mathbb{R}_{\ge 0}[\ell_i]$, with the exceptional divisor $E_i$ mapped to a (possibly singular) curve $C_i$ on $X_i$. 
Let $\ell_i\cong\mathbb{P}^1$ be the general fibre such that $K_X\cdot\ell_i=E_i\cdot \ell_i=-1$ (cf.~Lemma \ref{lem_del14}).
\item By Lemma \ref{lem_canonical_terminal}, each $X_i$ also has at worst isolated $\mathbb{Q}$-factorial Gorenstein canonical singularities  with $\rho(X_i)=\rho(X)-1$.
\item Let $I$ be the index recording the elementary $K_X$-negative extremal contractions.
\item Let $\nu:=\min\left\{\frac{L_X\cdot\ell_i}{-K_X\cdot\ell_i}~|~i\in I\subseteq\mathbb{N}\right\}=\min\{L_X\cdot\ell_i~|~i\in I\}$.
Then $D_X:=L_X+\nu K_X$ is nef.
In view of Theorem \ref{thm_canonical_k}, we may assume that $D_X\not\equiv 0$.
\item Let $I_0\subseteq I$ be the subset such that $i\in I_0$ if and only if $L_X\cdot \ell_i=\nu$.
\item Let $\varphi=\varphi_1:X\to X_1=:X'$, $E:=E_1$, $L_{X'}:=\varphi_*L_X$ and $C':=\varphi(E)$ with $1\in I_0$. 
\item Write $K_X=\varphi^*K_{X'}+E$, and $L_X=\varphi^*L_{X'}-\nu E$ (cf.~Lemma \ref{lem_del14}).
Let $D_{X'}:=L_{X'}+\nu K_{X'}$.
Then $D_X=\varphi^*D_{X'}$ and thus $D_{X'}$ is also nef.
\end{enumerate}	
\end{notation}

\textbf{Now we begin to prove Theorem \ref{thm_3fold_surface_curve}. 
If $\rho(X)\le 2$, then our theorem follows from Proposition \ref{contr_curve_elementary}.
Suppose that our theorem holds for the case $\rho(X)\le p$.
We shall assume $\rho(X)=p+1$ in the following.}

\begin{lemma}\label{lem_L.C>0}
Either $K_X+tL_X$ is ample for $t\gg 1$, or $L_{X'}\cdot C'\le 0$.
\end{lemma}
\begin{proof}
Suppose the contrary that $K_X+tL_X$ is not ample for any $t\gg 1$ and $L_{X'}\cdot C'>0$. 
Then, $L_{X'}$ is strictly nef on $X'$ and $K_{X'}+tL_{X'}$ is ample for  $t\gg 1$ by the induction on $\rho(X')$ (cf.~Remark \ref{remark_conclude_excep_tocurve}).
\textbf{We claim that $D_{X'}$ is semi-ample.} 
Consider the following 
$$\frac{2}{\nu}D_{X'}-K_{X'}=2(K_{X'}+\frac{1}{\nu}L_{X'})-K_{X'}=(K_{X'}+\frac{3}{2\nu}L_{X'})+\frac{1}{2\nu}L_{X'}.$$
Since $K_{X'}+\frac{3}{2\nu}L_{X'}=\frac{1}{\nu}D_{X'}+\frac{1}{2\nu}L_{X'}$ with $D_{X'}$ and $L_{X'}$ being nef, it must be big, for otherwise, $D_{X'}^3=D_{X'}^2\cdot L_{X'}=D_{X'}\cdot L_{X'}^2=L_{X'}^3=0$ would imply  $(K_{X'}+tL_{X'})^3=0$, a contradiction.
Hence, $\frac{2}{\nu}D_{X'}-K_{X'}$ is nef and big.
By the base-point-free theorem (cf.~\cite[Theorem 3.3]{KM98}), 
$D_{X'}$ is semi-ample. 
By Notation \ref{not_bir_surface_curve} (6) and Proposition \ref{prop-q-effective}, our $K_X+tL_X$ is ample for $t\gg 1$,  a contradiction.
\end{proof}

\begin{lemma}\label{lem_D.C>0}
Either $K_X+tL_X$ is ample for $t\gg 1$, or $D_{X'}\cdot C'>0$.	
\end{lemma}

\begin{proof}
Suppose the contrary that $K_X+tL_X$ is not ample for any $t\gg 1$, and $D_{X'}\cdot C'=0$. 
Then $K_{X'}\cdot C'\ge 0$ (cf.~Lemma \ref{lem_L.C>0}); hence $(K_X-E)|_E=\varphi^*K_{X'}|_E$ is a nef divisor on $E$ (cf.~Lemma \ref{lem_del14} and Notation \ref{not_bir_surface_curve}).
Besides, $D_{X'}\cdot C'=0$ implies that our $D_X|_E\equiv 0$;  hence $L_X|_E\equiv-\nu K_X|_E$. 
So our $-K_X|_E$ is strictly nef on $E$.  
Together with the nefness of $(K_X-E)|_E$,
our $-E|_E$ is also strictly nef on $E$. 
We consider the following  diagram.
\[\xymatrix@C=5em{
\widetilde{E}\ar[r]^\sigma\ar[d]_\tau\ar[dr]^q&E_N\ar[r]^{p_N}\ar[d]^{n_E}&C'_N\ar[d]^{n_{C'}}\\
\widetilde{E}_m\ar[d]&E\ar[r]^p\ar@{^(->}[d]&C'\ar@{^(->}[d]\\
C'_N&X\ar[r]^\varphi&X'
}
\]
Here, $p:=\varphi|_E$, $n_\bullet$ is the normalization, $\sigma$ is the minimal resolution, $\tau$ is the MMP of $\widetilde{E}$, $p_N$ is the induced map and $q:=n_E\circ\sigma$.
Then, $\widetilde{E}_m$ is a ruled surface. 
Denote by $C_0\subseteq \widetilde{E}_m$ the section with the minimal self-intersection $(C_0)^2=-e$, and
$f$ a general fibre of $\widetilde{E}_m$.

As is shown in  Proposition \ref{prop_Q-Goren_surface}, $K_{\widetilde{E}}\sim_{\mathbb{Q}}q^*K_E-F$, where $F$ is an effective divisor with $\varphi(q(F))\subseteq \textup{Sing}\,(C')$ (cf.~\cite[Lemma 5-1-9]{KMM87}, \cite[(4.1)]{Sak84} and Lemma \ref{lem_del14}).
Let $\widetilde{C}\subseteq \widetilde{E}$ be the strict transform of $C_0$  and $C_X:=q(\widetilde{C})$.
Since $C_X\not\subseteq\textup{Sing}\,E$ (cf.~Lemma \ref{lem_del14}),
$q$ is isomorphic in the generic point of $C_X$ and  $\widetilde{C}\not\subseteq\textup{Supp}\,F$. 
Then we have
\begin{equation}\label{equ_deform}
K_{\widetilde{E}}\cdot \widetilde{C}\le K_E\cdot C_X=(K_X+E)\cdot C_X\le -2,	
\end{equation}
since both $K_X$ and $E$ are Cartier (cf.~Lemma \ref{lem_canonical_terminal}). 
By \cite[Chapter \Rmnum{2}, Theorem 1.15]{Kol96},  
$$\dim_{\widetilde{C}}\textup{Chow}(\widetilde{E})\ge -K_{\widetilde{E}}\cdot \widetilde{C}-\chi(\mathcal{O}_{\widetilde{C}_N})\ge -K_{\widetilde{E}}\cdot \widetilde{C}-1\ge1,$$
where  $\widetilde{C}_N$ is the normalization of $\widetilde{C}$.
Then, $\widetilde{C}$ (and hence its push-forward $\tau_*(\widetilde{C})=C_0$) deforms, which  implies $e\le 0$.
Let us consider the following equalities, noting that $-K_X\cdot \ell=-E\cdot \ell=1$ for a general fibre $\ell\subseteq E$ of $\varphi$.
\begin{align*}
	&q^*(-K_X|_E)=\tau^*(C_0+\alpha f)+\sum a_iP_i,\\
	&q^*(-E|_E)=\tau^*(C_0+\beta f)+\sum b_iP_i,\\
	&q^*(-K_E)=\tau^*(2C_0+(\alpha+\beta)f)+\sum(a_i+b_i)P_i,\\
	&q^*((K_X-E)|_E)=\tau^*((\beta-\alpha)f)+\sum(b_i-a_i)P_i,
\end{align*}
with $P_i$ being $\tau$-exceptional.
On the one hand, the strict nefness of $-K_X|_E$ and $-E|_E$ gives  that $q^*(-K_X|_E)\cdot \tau^*(C_0)>0$ and $q^*(-E|_E)\cdot \tau^*(C_0)>0$; hence $\alpha-e>0$ and $\beta-e>0$.
On the other hand, since $\varphi^*K_{X'}|_E=(K_X-E)|_E$ is nef, we have $\beta-\alpha=q^*((K_X-E)|_E)\cdot \tau^*(C_0)\ge \varphi^*K_{X'}|_E\cdot C_X\ge 0$  \textbf{(\dag)}.

Since $K_{\widetilde{E}}\sim_{\mathbb{Q}}q^*K_E-F$,
there is an inclusion of the canonical sheaf $\omega_{\widetilde{E}}\subseteq q^*(\omega_E)$ as shown in Proposition \ref{prop_Q-Goren_surface} (cf.~ \cite[Lemma 5-1-9]{KMM87} and \cite[(4.1)]{Sak84}); hence  we have 
$\omega_{\widetilde{E}_m}=\tau_*(\omega_{\widetilde{E}})\subseteq (\tau_*q^*(\omega_E))^{\vee\vee}$. 
Consequently, we get the following inequality 
$$-2C_0-(\alpha+\beta)f\ge -2C_0+(2g-2-e)f,$$ 
with $g=g(C_0)$ being the genus.  
So $\alpha+\beta+2g-2-e\le 0$ \textbf{(\dag\dag)}.
Furthermore,  $C_0+\alpha f$  being  strictly nef on $\widetilde{E}_m$ by the projection formula, 
 we  have $(C_0+\alpha f)^2\ge0$. 
This gives that $\alpha\ge e/2$.
Together with \textbf{(\dag)} and \textbf{(\dag\dag)}, our $g\le 1$.

It is known that strictly nef divisors on (minimal) ruled surfaces over curves of genus $\le 1$ are indeed ample (cf.\,e.g.\,\cite[Example 1.23 (1)]{KM98}).
Therefore, $(C_0+\alpha f)^2>0$, and thus $\beta\ge \alpha>e/2$ (cf.~\textbf{(\dag)}).
Return back to \textbf{(\dag\dag)}, we have $g<1$.
So $g=0$ and $C_0\cong\mathbb{P}^1$.
Since $C_0$ deforms, we have $e=0$, $\widetilde{E}_m\cong\mathbb{P}^1\times\mathbb{P}^1$ and $\alpha+\beta\le 2$ (cf.~\textbf{(\dag\dag)}). 
Since  $K_X$ is Cartier, both $\alpha$ and $\beta$ are positive integers; hence $\alpha=\beta=1$.
Then we have $\textbf{(*)}:$ $\varphi^*K_{X'}|_E=(K_X-E)|_E\equiv 0$ (cf.~\textbf{(\dag)}), and thus $\varphi^*L_{X'}|_E\equiv 0$.

\textbf{We claim  that $\tau$ is an isomorphism.} 
Suppose  $\tau^*(C_0)=\widetilde{C}+\sum P_i$ with $P_i$ being $\tau$-exceptional.
Here, our $C_0$ can be chosen as any horizontal section containing some blown-up points of $\tau$ since $\widetilde{E}_m\cong\mathbb{P}^1\times\mathbb{P}^1$.
Then it follows from the projection formula that $\widetilde{C}^2<C_0^2=0$, a contradiction to  Equation (\ref{equ_deform}). 
So $\tau$ is isomorphic as claimed.

Now that $L_{X'}$ is nef, for $t\gg 1$, $K_{X'}+tL_{X'}$ is nef by the projection formula, noting that $\varphi^*(K_{X'}+tL_{X'})=(K_X+tL_X)+(\nu t-1)E$ and $\varphi^*K_{X'}|_E\equiv \varphi^*L_{X'}|_E\equiv 0$ as shown in $\textbf{(*)}$.
If $K_{X'}+tL_{X'}$ is  big for $t\gg 1$, then 
with the same proof of Lemma \ref{lem_L.C>0}, our $K_X+tL_X$ is ample, contradicting our assumption.
So we have $(K_{X'}+tL_{X'})^3=0$ for any $t\gg 1$. 
This in turn implies $K_{X'}^3=0$. 
On the other hand, we note that $K_X^3=0$ (cf.~Lemma \ref{lem-big-ample}). 
So we  get a contradiction (cf.~$\textbf{(*)}$ and Lemma \ref{lem_canonical_terminal} (6)):
$$0=K_X^3=(\varphi^*K_{X'}+E)^3=E^3=(q^*(E|_E))^2=(C_0+f)^2=2.$$
So our lemma is proved.
\end{proof}

\begin{lemma}\label{lem_intersect_C_<-1}
Suppose that $K_X+tL_X$ is not ample for any $t\gg 1$.
Then, for any curve $B'\subseteq X'$ such that $D_{X'}\cdot B'=0$, we have $K_{X'}\cdot B'\le -1$.
In particular, if $B'\cap C'\neq\emptyset$ (as sets), then $K_{X'}\cdot B'\le -2$.
\end{lemma}

\begin{proof}
By Lemma \ref{lem_D.C>0}, $B'\neq C'$.
Denote by $\hat{B'}$ the $\varphi$-proper transform of $B'$ on $X$.	
Then $D_X\cdot \hat{B'}=0$ and thus $K_{X}\cdot \hat{B'}\le -1$, noting that $L_X$ is strctly nef and  $K_X$ is Cartier (cf.~Lemma \ref{lem_canonical_terminal}).
Since $E\cdot \hat{B'}\ge 0$, we see that $K_{X'}\cdot B'=(K_X-E)\cdot \hat{B'}\le -1$.
In particular, if $B'\cap C'\neq\emptyset$, then $E\cdot \hat{B'}\ge 1$, and hence $K_{X'}\cdot \hat{B'}\le -2$. 
\end{proof}

\begin{lemma}\label{lem_not_big_X'}
Either $K_X+tL_X$ is ample for $t\gg 1$, or $D_{X'}$ is not strictly nef.
In the latter case, there is an extremal ray $\mathbb{R}_{\ge 0}[\ell']$ on $X'$ such that $K_{X'}\cdot\ell'<0$ and $D_{X'}\cdot \ell'=0$.
\end{lemma}

\begin{proof}
We may assume that $K_X+tL_X$ is not ample for any $t\gg 1$. 
Suppose that $D_{X'}$ is  strictly nef.
By induction (and Remark \ref{remark_conclude_excep_tocurve}), $K_{X'}+t_0D_{X'}$ is ample for any (fixed) $t_0\gg 1$.
But then, $K_X+t_0D_X=\varphi^*(K_{X'}+t_0D_{X'})+E$ is big; thus $K_X+uL_X=\frac{1}{1+t_0\nu}(K_X+t_0D_X)+(u-\frac{t_0}{t_0\nu+1})L_X$ is also big for  $u\gg 1$, a contradiction (cf.~Lemma \ref{lem-not-big-some}). 
Hence, there is an irreducible curve $B'\in\overline{\textup{NE}}(X')$ such that $D_{X'}\cdot B'=0$.
By Lemma \ref{lem_intersect_C_<-1},
$K_{X'}\cdot B'<0$.
Since $D_{X'}$ is nef, by the cone theorem (cf.~\cite[Theorem 3.7]{KM98}), there exists a $K_{X'}$-negative extremal curve $\ell'$ such that $K_{X'}\cdot\ell'<0$ and $D_{X'}\cdot\ell'=0$. 
\end{proof}

\begin{lemma}\label{lem-sec-con-surface}
If $K_X+tL_X$ is not ample for any $t\gg 1$, then the contraction $\varphi':X'\to X''$ of $\mathbb{R}_{\ge 0}[\ell']$ in Lemma \ref{lem_not_big_X'} is birational.	
\end{lemma}

\begin{proof}
Suppose the contrary that $\dim X''\le 2$.
By the cone theorem (cf.~\cite[Theorem 3.7]{KM98}), $D_{X'}=\varphi'^*(D_{X''})$ for some nef divisor $D_{X''}$ on $X''$.
If $X''$ is a point, then $\rho(X')=1$ and $D_{X'}$ is ample, a contradiction to Proposition \ref{prop-q-effective}.
If $X''$ is a curve or  $\dim X''=2$ and $D_{X''}^2\neq 0$ (and hence big),
then $D_{X'}$ (and hence $D_X$) is numerically equivalent to a non-zero effective divisor, contradicting  Proposition \ref{prop-q-effective} again.

Suppose that $\dim X''=2$ and $D_{X''}^2=0$. 
We claim that the composite $\varphi'\circ\varphi:X\to X''$ is an equi-dimensional $K_X$-negative contraction (of an extremal face); thus we get a contradiction to our assumption by Theorem \ref{thm_contr_surface_conic}. 
First, by the cone theorem (cf.~\cite[Theorem 3.7]{KM98}), $\varphi$ is equi-dimensional (of relative dimension one).
Since $C'$ is not contracted by $\varphi'$ (cf.~Lemma \ref{lem_D.C>0}), the composite $\varphi'\circ\varphi$ is also equi-dimensional. 
Take any irreducible curve $F$ contracted by $\varphi'\circ\varphi$.
Then $D_X\cdot F=(\varphi'\circ\varphi)^*(D_{X''})\cdot F=0$ and thus $K_X\cdot F<0$.
Hence, $-K_X$ is $(\varphi'\circ\varphi)$-ample  (cf.~\cite[Theorem 1.42]{KM98}) and 
$\varphi'\circ\varphi$ is a $K_X$-negative contraction.
Finally, the contraction $X\to X''$ is  clearly extremal by 
considering the pullback of any ample divisor on $X''$. 
\end{proof}

The following lemma is a bit technical. 
We divide the proof into several cases for readers' convenience. 
We shall heavily apply the terminalization and Lemma \ref{lem_canonical_terminal}.
Recall that the \textit{length} of a $K_X$-negative extremal contraction is defined to be the minimum of $-K_X\cdot B$ for generic curves $B$ in the covering families of contracted locus.

\begin{lemma}\label{lem_bir-intersectC}
If $K_X+tL_X$ is not ample for any $t\gg 1$, then the contraction $\varphi':X'\to X''$ of $\mathbb{R}_{\ge 0}[\ell']$ in Lemma \ref{lem_not_big_X'} is  divisorial  with the exceptional divisor  $E'$ such that   $E'\cap C'=\emptyset$.
\end{lemma}

\begin{proof}
By Lemmas \ref{lem-sec-con-surface} and \ref{lem_del14},  our $\varphi'$ is  a divisorial contraction.
In the following, we shall discuss case-by-case in terms of $E'$ and the intersection of $E'\cap C'$.
\par \vskip 0.4pc \noindent
\textbf{Case (1). Suppose $C'\subseteq E'$.}  
Then $C'$ being rigid and $D_{X'}\cdot C'>0$ (cf.~Lemmas \ref{lem_L.C>0} and \ref{lem_D.C>0}) would imply that $\varphi'$ is a blow-up of a curve on $X''$ (cf.~Lemma \ref{lem_canonical_terminal}) with $C'$ being  horizontal  on $E'\subseteq X'$. 
Let $\ell'$ be a general fibre of $\varphi'$. 
Since $D_{X'}\cdot \ell'=0$ and $\ell' \cap C'\neq\emptyset$, we have $K_{X'}\cdot\ell'\le -2$ (cf.~Lemma \ref{lem_intersect_C_<-1}),  contradicting Lemma \ref{lem_del14} (3).

\par \vskip 0.4pc \noindent
\textbf{Case (2). Suppose  $E'\cap C'$ is a finite non-empty  set.} 
Then by Lemma \ref{lem_canonical_terminal}, we have the following commutative diagram
\begin{align}\label{diagram_4}\tag{$*$}
\xymatrix{Y'\ar[r]^{\phi'}\ar[d]_{\tau'}&Y''\ar[d]^{\tau''}\\
X'\ar[r]_{\varphi'}&X''
}	
\end{align}
where $\tau'$ is the crepant terminalization, $\phi'$ is a divisorial contraction, and $Y'$ has only $\mathbb{Q}$-factorial Gorenstein terminal singularities.

\par \vskip 0.4pc \noindent
\textbf{(2i).} Suppose $\phi'$ maps the exceptional divisor $E_{Y'}$ to a point  and the length of  $\phi'$ is one (this is the case when \cite[Theorem 5 (2), (3) or (4)]{Cut88} happen). 
By Lemma \ref{lem_canonical_terminal}, $E'\cong E_{Y'}$.  
Then we can pick a  
curve $\ell'\subseteq E'\subseteq X'$ of $\varphi'$ meeting $C'$ such that $K_{X'}\cdot\ell'=-1$ (noting that $\tau'$ is crepant) and $D_{X'}\cdot \ell'=0$. 
However, this contradicts Lemma \ref{lem_intersect_C_<-1}.

\par \vskip 0.4pc \noindent 
\textbf{(2ii).} Suppose  $\phi'$ maps the exceptional divisor $E_{Y'}$ onto a curve   (and hence the length of  $\phi'$ is still one, which is the case when \cite[Theorem 4]{Cut88} happens).
Let $P\in E'\cap C'$.
If $P$ is a terminal point, then we 
pick a fibre $\ell_{Y'}\subseteq E_{Y'}$ passing through $\tau'^{-1}(P)$.
If $P$ is not a terminal point, then we take a curve $c_0\subseteq \tau'^{-1}(P)\cap E_{Y'}$, which is a horizontal curve of $\phi'$; in this case, we pick $\ell_{Y'}$ to be any fibre of $\phi'$, which automatically intersects with $c_0$.
In both cases, let $\ell':=\tau'(\ell_{Y'})\ni P$.  
Similarly, $K_{X'}\cdot\ell'=\tau'^*K_{X'}\cdot \ell_{Y'}=K_{Y'}\cdot \ell_{Y'}=-1$, which contradicts Lemma \ref{lem_intersect_C_<-1} again.

\par \vskip 0.4pc \noindent
\textbf{(2iii).} By the classification, we may assume that $\phi'$ maps the exceptional divisor $E_{Y'}$ to a point and the length of $\phi'$ is two (this is the case when \cite[Theorem 5 (1)]{Cut88} happens). 
In this case, $E_Y'\cong\mathbb{P}^2$ with $E_{Y'}|_{E_{Y'}}=\mathcal{O}(-1)$, $Y''$ (and hence $Y'$) is  smooth, and $\varphi'(E')$ is a point. 
Consider the following  (by taking $\tau_2:=\tau'$ and $\tau_3:=\tau''$ in the diagram (\ref{diagram_4})):
\begin{align}\label{diagram_5}\tag{$**$}
\xymatrix{Z\ar[d]_{\tau_0}&Y\ar[r]^{\phi}\ar[l]_{\psi_Y}\ar[d]_{\tau_1}&Y'\ar[r]^{\phi'}\ar[d]^{\tau_2}&Y''\ar[d]^{\tau_3}\\
W&X\ar[l]^{\psi_X}\ar[r]_\varphi &X'\ar[r]_{\varphi'}&X''
}	
\end{align}

\textbf{We first explain in the following how to get the  diagram (\ref{diagram_5}).}
By Lemma \ref{lem_canonical_terminal}, we get the diagrams $\tau_2\circ\phi=\varphi\circ\tau_1$ and $\tau_3\circ\phi'=\varphi'\circ\tau_2$,
noting that $\textup{Exc}(\tau_1)$ is disjoint with $E_Y:=\textup{Exc}(\phi)$ and $\textup{Exc}(\tau_2)$ is  disjoint with $E_{Y'}:=\textup{Exc}(\phi')$.
This implies that $E'\cong E_{Y'}=\mathbb{P}^2$ and  $X'$ is smooth around $E'$.
Fixing a line $\ell'\subseteq E'$ of $\varphi'$ meeting $C'$ and taking $\hat{\ell'}$ to be its proper transform on  $X$,  we have
$K_{X'}\cdot \ell'=-2$.
Let $a:=\hat{\ell'}\cdot E\ge 1$. 
Then $K_X\cdot\hat{\ell'}=-2+a$ and $L_X\cdot\hat{\ell'}=\nu(2-a)>0$.
In particular, $a=1$ and $K_X\cdot\hat{\ell'}=-1$.
\textbf{This also implies that $E'$  meets $C'$ along a  single (smooth) point of $C'$.}
Indeed, if $\sharp (E'\cap C')\ge 2$, then the line $\ell'\subseteq E'$ passing through any two points of $E'\cap C'$ satisfies $E\cdot\hat{\ell'}\ge 2$, which is absurd. 
So the $\varphi$-proper transform $\hat{E'}\cong\mathbb{F}_1$ of $E'$ is ruled over $h_W\cong\mathbb{P}^1$ with fibres $\hat{\ell'}$ and the negative section $C_0$ (being a fibre of $\varphi$). 
Suppose $P\in E'\cap C'$ is a singular point  of $C'$.  
Then   $E'\cdot C'=\deg \mathcal{O}_{X'}(E')|_{\widetilde{C'}}\ge 2$, with $\widetilde{C'}$  the normalization of $C'$, in which case,
any line $\ell'\subseteq E'$ containing $P$ satisfies  $E\cdot\hat{\ell'}=(E|_{\hat{E'}}\cdot\hat{\ell'})_{\hat{E'}}=(C'\cdot E')\ge 2$, a contradiction.
So  $E'$ meets $C'$ along a single smooth point of $C'$.

Since $\hat{E'}\cdot\hat{\ell'}=\varphi^*E'\cdot\hat{\ell'}=-1$ (noting that $E'|_{E'}=\mathcal{O}(-1)$) and  $-\hat{E'}|_{\hat{E'}}$ is  relative ample with respect to the ruling $\hat{E'}\to h_W$, we obtain a contraction $\psi_X:X\to W$   such that $\psi_X|_{\hat{E'}}$ coincides with this ruling and $\psi_X|_{X\backslash\hat{E'}}\cong W\backslash h_W$ (cf.~e.g. \cite[Proposition 7.4]{HP16}). 
Similar to $\psi_X$, the morphism $\psi_Y$ is induced by contracting the divisor $\hat{E}_{Y'}:=\tau_1^*\hat{E'}\cong\mathbb{F}_1$ to $h_Z\cong h_W\cong\mathbb{P}^1$. 
By the rigidity lemma, $\psi_X\circ\tau_1$ factors through $\psi_Y$ 
and we get $\tau_0$.

\textbf{Caution: it is still not clear whether $W$ is projective or not, and then $\psi_X$ here may not be an extremal contraction.
Therefore, we could not apply the induction on $W$ so far.} 

Since $X$ has only canonical singularities and $L_X+\nu\hat{E'}$ is $\psi_X$-trivial,  
by \cite[Theorem 4.12]{Nak87}, there exists a  divisor  $L_W$ on $W$ such that
$L_X=\psi_X^*(L_W)-\nu\hat{E'}$. 
Recall that $C_0$ (a fibre of $\varphi$) is the negative section of $\hat{E}'$. 
Since $-{\hat{E'}}|_{{\hat{E'}}}=C_0+\hat{\ell'}$ (noting that $\hat{E'}\cdot\hat{\ell'}=E'\cdot\ell'=-1$ and $\hat{E'}\cdot C_0=0$), we obtain the following strictly nef divisor on $\hat{E'}$
$$L_X|_{\hat{E}'}=\nu C_0+(L_W\cdot h_W+\nu)\hat{\ell'}.$$
Since  a strictly nef divisor on $\mathbb{F}_1$ is ample, our $L_X|_{\hat{E}'}$ is ample and hence  $L_W\cdot h_W>0$ (cf.~\cite[Chapter \Rmnum{5}, Proposition 2.20]{Har77}). 
Then $L_W$ is strictly nef on the Moishezon threefold $W$ by the projection formula. 
We consider the following commutative diagram:
\[\xymatrix{
Z_0\ar[d]_{\sigma_0}\ar@/_2pc/[dd]_{\pi_0}&Y_0\ar[l]_{\psi_{Y_0}}\ar[d]^{\sigma_1}\ar@/^2pc/[dd]^{\pi_1}\\
Z\ar[d]_{\tau_0}&Y\ar[d]^{\tau_1}\ar[l]_{\psi_Y}\\
W&X\ar[l]^{\psi_X}
}
\]
where $\sigma_1:Y_0\to Y$ is a resolution with $Y_0$ being smooth.
Since $X$ is smooth around $\hat{E'}$ (recalling that $E'\cap C'$  is a smooth point of $C'\subseteq X'$), our $Y$ is also smooth around $\hat{E}_{Y'}$.
So  $\sigma_1$ is isomorphic around $\hat{E}_{Y'}$.
In particular, $Y_0$ admits a contraction $\psi_{Y_0}$ mapping $E_{Y_0}:=\sigma_1^*(\hat{E}_{Y'})=\pi_1^*\hat{E'}\cong\mathbb{F}_1$ onto a curve $h_{Z_0}\cong\mathbb{P}^1$ (cf.~e.g. \cite[Proposition 7.4]{HP16}).
By the rigidity lemma, $\psi_Y\circ\sigma_1$ factors through $\psi_{Y_0}$, and we get the induced $\sigma_0$.
Since  $Y_0$ is smooth and the conormal sheaf of $E_{Y_0}$ is isomorphic to $-E_{Y_0}|_{E_{Y_0}}$ which is (locally over $h_{Z_0}$) isomorphic to $\mathcal{O}(1)$, 
it follows from \cite[Corollary 6.11]{Art70} that $Z_0$ is smooth.
\begin{claim}\label{claim_proj}
$Z_0$ is projective.	
\end{claim}

Suppose the claim for the time being. 
Since $-K_X$ is $\psi_X$-ample, we have $R^j(\psi_X)_*\mathcal{O}_X=0$ for all $j\ge 1$ (cf.~\cite[Theorem 1-2-5]{KMM87}) and  hence $W$ has only isolated rational singularities.
Since $Z_0\to W$ is a resolution and $Z_0$ is projective by the assumption, 
by \cite[Remark 3.5]{HP16}, $W$ admits a (smooth) K\"ahler form and hence $W$ is both K\"ahler and Moishezon.
So the projectivity of $W$ follows from \cite[Theorem 6]{Nam02}.
Now, $W$ being projective and $L_W$ being strictly nef, we get a contradiction by (Proof of) Lemma \ref{lem_L.C>0}.

\par \vskip 1pc \noindent
\textbf{Proof of Claim \ref{claim_proj} (End of Proof of Lemma \ref{lem_bir-intersectC}).}
Suppose the contrary. 
Denote by $\pi_i:=\tau_i\circ\sigma_i$.
By  \cite[Theorem 2.5]{Pet96}, there is an irreducible curve $b$ and a positive closed current $T$ on $Z_0$ such that $b+T\equiv0$ (as $(2,2)$-currents); hence $(\pi_0)_*(b+T)\equiv0$.
Since $L_W$ is strictly nef on $W$, our $b$ is $\pi_0$-exceptional.
Note that $X$ is smooth around $\hat{E'}$ and $\hat{E}_{Y'}\cap \textup{Exc}(\tau_1)=\emptyset$.
So $E_{Y_0}\cap\textup{Exc}(\pi_1)=\emptyset$.
Take a very ample divisor $H$ on $Y_0$.
Since $Z_0$ is $\mathbb{Q}$-factorial, it is easy to verify that $(\psi_{Y_0})_*H\cdot b>0$ (cf.~\cite[Lemma 2.62]{KM98}), noting that $\pi_0^*L_W\cdot h_{Z_0}>0$   and thus $b\neq h_{Z_0}:=\psi_{Y_0}(E_{Y_0})$.
So  $(\psi_{Y_0})_*H\cdot T<0$.
By \cite{Siu74}, our $T=\chi_{h_{Z_0}}T+\chi_{Z_0\backslash h_{Z_0}}T=\delta h_{Z_0}+\chi_{Z_0\backslash h_{Z_0}}T$, where $\chi_{h_{Z_0}}T$ and $\chi_{Z_0\backslash h_{Z_0}}T$ are positive closed currents. 
Since $(\psi_{Y_0})_*H\cdot \chi_{Z_0\backslash h_{Z_0}}T\ge 0$, we have $\delta>0$.
Then
$$0\equiv (\pi_0)_*(b+T)\equiv (\pi_0)_*T=\delta h_W+(\pi_0)_*(\chi_{Z_0\backslash h_{Z_0}}T).$$
 Since $L_W\cdot h_W>0$ and $L_W\cdot(\pi_0)_*(\chi_{Z_0\backslash h_{Z_0}}T)\ge 0$, the above equality is absurd.
\end{proof}

\begin{proof}[\textup{\textbf{End of Proof of Theorem \ref{thm_3fold_surface_curve}.}}]
We suppose the contrary that $K_X+tL_X$ is not ample for any $t\gg 1$. 
By Lemmas \ref{lem-sec-con-surface} and \ref{lem_bir-intersectC}, 
 $\varphi':X'\to X''$ is a divisorial contraction with the exceptional divisor $E'$  disjoint with $C'$.
Then the strict transform of $E'$ on $X$ is some $E_j$, $j\in I_0$ (cf.~Notation \ref{not_bir_surface_curve}); hence we can continue to consider $D_{X''}$.
By Proposition \ref{prop-q-effective} and the induction  on $X''$ (cf.~Remark \ref{remark_conclude_excep_tocurve}),  our $D_{X''}$ is  not strictly nef. 
So we get the third  contraction $X''\to X'''$ (cf.~Lemma \ref{lem_not_big_X'}). 
If $\dim X'''\le 2$, then we  argue as in Lemma \ref{lem-sec-con-surface}, together with Theorem \ref{thm_contr_surface_conic}, to conclude that $K_X+tL_X$ is ample, a contradiction to our assumption.   
So $X''\to X'''$ is still birational with the exceptional divisor $E''$. 
But then, Lemma \ref{lem_bir-intersectC} Case (1) shows that neither $\varphi'(C')$ (where $C'=\varphi(E)\subseteq X'$) nor $C''$ (where $C''=\varphi'(E')\subseteq X''$) is contained in $E''$.
Further,  Lemma \ref{lem_bir-intersectC} Case (2) gives that
 such $E''$ cannot intersect with  $\varphi'(C')\cup C''$. 
Hence,  $E$, $E'$ and $E''$ are pairwise disjoint. 
Since  $X$ is uniruled, after finitely many steps, we get some $X_M$ with $\dim X_M\le 2$. 
By Theorem \ref{thm_contr_surface_conic}, $K_X+tL_X$ is ample which contradicts our assumption (cf.~Lemma \ref{lem-sec-con-surface}).
\end{proof}

\begin{proof}[Proof of Theorems \ref{main_Goren_ter_3fold} and  \ref{main_theorem_Goren_ter_3fold} (3)]
It follows  from Remark \ref{remark_conclude_excep_tocurve} and 	Theorem \ref{thm_3fold_surface_curve}.
\end{proof}

\begin{remark}\label{rem_kappa=0}
In Theorem \ref{main_Goren_ter_3fold},	if $\kappa(X)=0$,  then things will become different.
In this case, one can start from  $K_X\equiv 0$ (cf.~Proposition \ref{prop-q-effective}) and then
 Question \ref{ques_serrano} predicts that strict nefness is equivalent to ampleness, which relates to the abundance conjecture.
 We refer to \cite{LS20} for a recent progress on the strictly nef divisors on Calabi-Yau threefolds.
\end{remark}

\section{Further discussions, Proof of Theorem  \ref{intro_main_prop}}\label{section_last}
In this section, we consider a projective klt pair $(X,\Delta)$ with strictly nef $-(K_X+\Delta)$. 
We first show that such $X$ has vanishing augmented irregularity (cf.~Theorem \ref{thm_num_trivial}), which  reduces Question \ref{gene_conj} (in any dimension) to $q^{\circ}(X)=0$. 
As an application, we shall prove Theorem \ref{intro_main_prop} 
which answers Question \ref{gene_conj} affirmatively for the threefold case.


\begin{theorem}\label{thm_num_trivial}
Let $(X,\Delta)$ be a projective klt pair with $-(K_X+\Delta)$ being strictly nef.
Then the augmented irregularity $q^\circ(X)=0$.	
\end{theorem}

The following theorem, which was first proved in \cite[Theorem A]{Wan20} for the Albanese map and then strengthened recently (cf.~\cite[Theorem 4.7]{MW21}), will be crucially used in the proof of Theorem \ref{thm_num_trivial}.
We refer readers to \cite[Definition 2.6]{MW21} for the definitions of locally constant fibration with respect to pairs (cf.~\cite[Definition 1.6]{Wan20}).
\begin{theorem}(cf.~\cite[Theorem A]{Wan20} and \cite[Theorem 4.7]{MW21})\label{thm_wan_m21}
Let $(X,\Delta)$ be a projective klt pair such that $-(K_X+\Delta)$ is nef.
Then the Albanese map $\pi:X\dashrightarrow A$ is an everywhere defined locally constant fibration with respect to $(X,\Delta)$.
\end{theorem}

The following lemma gives a sufficient condition for the Albanese map being stable under \'etale base change.
\begin{lemma}\label{lem_base_change_simply_connected}
Let $(X,\Delta)$ be a  projective klt  pair such that the anti-log canonical divisor $-(K_X+\Delta)$ is nef.
Let $\pi:X\to A$ be the Albanese map, and $A'\to A$ any \'etale cover.
Suppose that the general fibre $F$ of $\pi$ has $q(F)=0$.
Then the base change $X':=X\times_AA'\to A'$ is still the Albanese map.
\end{lemma}
\begin{proof}
We consider the following commutative diagram
\[\xymatrix{&\textup{Alb}(X')\ar[d]^\sigma\\
X'\ar[r]^{\pi'}\ar[ur]^{\textup{alb}_{X'}}\ar[d]&A'\ar[d]\\
X\ar[r]^\pi&A
}
\]
where $\sigma$ is induced from the universality of $\textup{alb}_{X'}$. 
Note that both $\pi$ and $\textup{alb}_{X'}$ are locally constant  surjective morphisms (cf.~Theorem \ref{thm_wan_m21}).
So we only need to show that $\sigma$ is isomorphic.
Since $\pi'$ is surjective, so is $\sigma$.
If $q(X')>q(X)$, then taking a general fibre of $\sigma$, it is dominated by a fibre $F$ of $\pi'$ with $q(F)=0$, which is absurd. 
So $\sigma$ is generically finite.
Since $\pi'$ has connected fibres, our $\sigma$ is birational and hence an isomorphism.
\end{proof}

Let us show Theorem \ref{thm_num_trivial}. 
Our main tool is to find a special section of the Albanese map by extending \cite[Theorem 1.3]{LOY19} to the singular setting.  
For the organization,  we  shall only highlight the differences 
and skip the same arguments during the proof.
We strongly recommend readers to \cite[Theorem 1.3 and Section 4]{LOY19} for the complete  arguments on the smooth case. 
\begin{proof}[Proof of Theorem \ref{thm_num_trivial}]
We shall use the induction on the dimension of $X$ to prove our theorem.	
If $\dim X=1$, then $X\cong\mathbb{P}^1$, in which case our theorem is trivial.
Suppose that our theorem holds for $\dim X\le k$ with $k\ge 2$.
In the following, we may assume that $\dim X\ge k+1$.
Let us suppose the contrary that $q^\circ(X)>0$.
Then with $X$ replaced by its quasi-\'etale cover, we may assume that $q(X)=q^\circ(X)>0$ (cf.~\cite[Proposition 5.20]{KM98}).
Denote by $\pi:X\to A:=\textup{Alb}(X)$ its Albanese map.
Since $-(K_X+\Delta)$ is nef, our $\pi$ is a locally constant surjective morphism (cf.~Theorem \ref{thm_wan_m21}). 
Besides, the strict nefness of $-(K_X+\Delta)$ implies that $X$ is uniruled and hence $\dim A<\dim X$ (cf.~Proposition \ref{prop_strict_uniruled}).
Denote by $F$ a fibre of $\pi$.
By our assumption and the adjunction, we have  $0<\dim F<\dim X$, $(F,\Delta|_F)$ is klt, and $-(K_F+\Delta|_F)$ is also strictly nef.
Applying the induction, we see that $q^\circ(F)=0$.
By Lemma \ref{lem_base_change_simply_connected}, we are free to replace $A$ by its \'etale cover. 
In the following, we shall construct a section $\sigma:A\to X$ (up to an \'etale cover) of $\pi$ such that $\sigma^*(K_X+\Delta)\equiv 0$, which contradicts the strict nefness of $-(K_X+\Delta)$.

By \cite[Section 3.4]{Wan20}, there is a $\pi$-very ample divisor $L$ on $X$ such that $E_m:=\pi_*\mathcal{O}_X(mL)$ is numerically flat for each $m\ge 0$.
Let $\widetilde{A}\to A$ be the universal cover with the Galois group $G:=\textup{Gal}(\widetilde{A}/A)$, and $\widetilde{X}:=X\times_A\widetilde{A}$ with the induced projection $p:\widetilde{X}\to X$.
By Theorem \ref{thm_wan_m21}, 
$\widetilde{X}\cong Y\times\widetilde{A}$ for some projective variety $Y$ (isomorphic to a fibre $F$ of $\pi$) which is stable under $G$, and $X\cong\widetilde{X}/G$ via the diagonal action.
Let $\textup{pr}_1:\widetilde{X}\to Y$ be the natural projection.
Further, since $\pi$ is locally constant with respect to the pair $(X,\Delta)$, there exists a $G$-stable Weil $\mathbb{Q}$-divisor $\Delta_Y$ on $Y$ such that $\Delta\cong \textup{pr}_1^*\Delta_Y/G=(\Delta_Y\times \widetilde{A})/G$  (cf.~\cite[Definition 2.6]{MW21}). 
In other words, $\textup{pr}_1^*\Delta_Y=\Delta_Y\times \widetilde{A}=p^*\Delta$ (as Weil divisors), noting that the pull-back is well-defined since both $p$ and $\textup{pr}_1$ are equi-dimensional (cf.~\cite[Construction 2.13]{CKT16}). 

With the same arguments as in \cite[Proof of Theorem 1.3]{LOY19}, there is a $G$-equivariant ample line bundle $H$ on $Y$ such that $p^*L=\textup{pr}_1^*H$. 
Moreover, the action  $G$ on $Y$ has a fixed point $y\in Y$ with $G$ replaced by a finite index subgroup (cf.~\cite[Theorem 4.1]{LOY19} and Lemma \ref{lem_base_change_simply_connected}). 
So it follows from \cite[Lemma 4.5]{LOY19} that $y$ induces a section $\sigma:A\to X$ of the Albanese morphism $\pi$, and there is a short exact sequence of flat vector bundles on $A$
$$0\to I\to E_1\to Q\to 0$$
such that $Q\cong\sigma^*L$.
Since $Q$ is a numerically flat line bundle, we have $\sigma^*L\equiv 0$.

\textbf{Now we  construct another numerically flat vector bundle.}
Since $(X,\Delta)$ is  klt  and $\pi$ is locally constant with respect to the pair $(X,\Delta)$ (cf.~Theorem \ref{thm_wan_m21}), for every fibre $F$, $(F,\Delta|_F)$ is also klt and hence $F$ is Cohen-Macaulay (cf.  \cite[Theorems 5.10 and 5.22]{KM98}).
So $\pi$ is a  Cohen-Macaulay morphism and thus the relative dualizing sheaf $\omega_{X/A}$ is compatible with any base change, i.e.,
$\omega_{\widetilde{X}/\widetilde{A}}=p^*\omega_{X/A}$. 
Since $\widetilde{X}\cong Y\times \widetilde{A}$, we see that $\omega_{\widetilde{X}/\widetilde{A}}=\omega_{\widetilde{X}}=\textup{pr}_1^*\omega_Y$. 
Let $t$  be a natural number such that both $t\Delta$ and $t\Delta_Y$ are integral Weil divisors, and both $K_X+\Delta$ and $K_Y+\Delta_Y$ are Cartier divisors. 
Since $K_Y+\Delta_Y$ is $G$-equivariant, 
for any $g\in G$, we have
$g^*(tK_Y+t\Delta_Y)=tK_Y+t\Delta_Y$.
Since $H$ is ample and also $G$-equivariant on $Y$, for  $k\gg 1$, our $\mathcal{O}_Y(tK_Y+t\Delta_Y)\otimes H^{\otimes k}$ is a $G$-equivariant ample line bundle. 
Now we have the following equality
\begin{align*}\label{equ_cohen_double}
p^*\mathcal{O}_X(tK_X+t\Delta)&=\mathcal{O}_{\widetilde{X}}(tp^*K_X+tp^*\Delta)=\mathcal{O}_{\widetilde{X}}(tK_{\widetilde{X}}+t(\Delta_Y\times\widetilde{A}))\\
&=\mathcal{O}_{\widetilde{X}}(\textup{pr}_1^*(tK_Y+t\Delta_Y))=\textup{pr}_1^*\mathcal{O}_Y(tK_Y+t\Delta_Y)
\end{align*}
Let $r$ be a sufficiently large positive integer such that $(\mathcal{O}_Y(tK_Y+t\Delta_Y)\otimes H^{\otimes k})^{\otimes r}$ is ($G$-equivariantly) very ample.
Then  the natural linear action $G$ on $H^0(Y,(\mathcal{O}_Y(tK_Y+t\Delta_Y)\otimes H^{\otimes k})^{\otimes r})$ induces a flat vector bundle structure on $E':=\pi_*(\mathcal{O}_X(tK_X+t\Delta)\otimes L^{\otimes k})^{\otimes r}$. 
By \cite[Lemma 4.5 and the paragraph before it]{LOY19}, the $G$-fixed point $y\in Y$  induces a short exact sequence of flat vector bundles:
$$0\to I'\to E'\to Q'\to 0$$
such that $0\equiv Q'\cong \sigma^*(\mathcal{O}_X(tK_X+t\Delta)\otimes L^{\otimes k})^{\otimes r}$, noting that $Q'$ is a line bundle.

Since $\sigma^*L\equiv 0$ as shown above, we see that $\sigma^*(K_X+\Delta)\equiv 0$, which contradicts the strict nefness of $K_X+\Delta$.
\end{proof}

With all the preparations settled, we are now ready to show Theorem  \ref{intro_main_prop}.

\begin{proof}[\textup{\textbf{Proof of Theorem \ref{intro_main_prop}}}]
Since $(X,\Delta)$ is klt and $-(K_X+\Delta)$ is nef, it follows from \cite[Corollary 1.2]{MW21} that there exists a finite quasi-\'etale cover $\nu:(X',\Delta':=\nu^*\Delta)\to (X,\Delta)$ such that $X'$ admits a locally constant MRC fibration $\phi:X'\to Y$ with $K_Y\equiv 0$ and  $Y\cong A\times\prod Y_i\times\prod Z_j$. 
Here, $A$ is an abelian variety, $Y_i$ are strict Calabi-Yau varieties, and $Z_j$ are singular holomorphic symplectic varieties. 
Since $-(K_{X'}+\Delta')$ is also strictly nef, it follows from 
Theorem \ref{thm_num_trivial} that  $q(X')=q(Y)=0$. 
Hence, the abelian variety part $A$ in the  decomposition of $Y$ does not appear.

Suppose the contrary that $X$ is not rationally connected.  
Then  $\dim Y=1$ or $2$ (cf.~Proposition \ref{prop_strict_uniruled}).
If $\dim Y=1$, then $q(Y)=0$ forces $Y$ to be  rational, a contradiction to the MRC fibration.
So $\dim Y=2$ and thus $Y$ is simply connected (cf.~\cite[Corollary 13.3]{GGK19}). 
Since $\phi$ is locally constant, 
we have $X'=C\times Y\cong\mathbb{P}^1\times Y$ with the induced  projection $\pi_1:X'\to C$. 
Then $-(K_{X'}+\Delta')=-\pi_1^*K_C-\Delta'$ is strictly nef. 
Taking  any curve $\ell$ contracted by $\pi_1$, we have $-\Delta'\cdot \ell>0$, and then
 $\textup{Supp}\,\Delta'$ dominates $C$. 
 By taking a $\pi_1$-contracted (general) curve which is not contained in $\textup{Supp}\,\Delta'$, we get a contradiction.
Our theorem is proved. 
\end{proof}

We end up this section (and also the paper) with the following remark.
\begin{remark}
In view of Theorem \ref{thm_num_trivial}, \cite[Corollary 1.2]{MW21} and the above proof, to answer Question \ref{gene_conj} in higher dimension, we are only left to show the finiteness of the fundamental group of such $Y$ (in the above proof) with  vanishing augmented irregularity $q^\circ(Y)$ and numerically trivial canonical divisor $K_Y$ (cf.~\cite[Conjecture 1.4]{MW21}).
\end{remark}

\end{document}